\documentclass[a4paper,11pt]{article}

\usepackage[T1]{fontenc}
\usepackage{lmodern}

\usepackage{dsfont}

\usepackage[latin1]{inputenc}
\usepackage{amsmath}
\usepackage{amsthm}
\usepackage{amssymb}
\usepackage{mathrsfs}
\usepackage{graphicx}
\usepackage[all]{xy}
\usepackage{hyperref}

\usepackage{makeidx}

\usepackage{stmaryrd}
\usepackage{caption}

\usepackage{abstract} 

\newtheorem{thm}{Theorem}[section]
\newtheorem{cor}[thm]{Corollary}
\newtheorem{claim}[thm]{Claim}
\newtheorem{fact}[thm]{Fact}

\newtheorem{lemma}[thm]{Lemma}
\newtheorem{prop}[thm]{Proposition}

\theoremstyle{definition}
\newtheorem{definition}[thm]{Definition}

\newtheorem{remark}[thm]{Remark}

\def\rquotient#1#2{%
	\makeatletter
	\raise.3ex\hbox{$#1$}/\lower.3ex\hbox{$#2$}%
	\makeatother
}	

\makeatletter
\newcommand{\subjclass}[2][2010]{%
	\let\@oldtitle\@title%
	\gdef\@title{\@oldtitle\footnotetext{#1 \emph{Mathematics subject classification.} #2}}%
}
\newcommand{\keywords}[1]{%
	\let\@@oldtitle\@title%
	\gdef\@title{\@@oldtitle\footnotetext{\emph{Key words and phrases.} #1.}}%
}
\makeatother

\newcommand{\Address}{{
		\bigskip
		\small
		
		\textsc{University of Montpellier\\ 
Institut Math\'ematiques Alexander Grothendieck\\
Place Eug\`ene Bataillon\\
34090 Montpellier (France}\par\nopagebreak
		\textit{E-mail address}: \texttt{anthony.genevois@umontpellier.fr}
		
}}

\makeindex

\title{Translation lengths in crossing and contact graphs of quasi-median graphs}
\date{\today}
\author{Anthony Genevois}


\begin{document}

\maketitle

\begin{abstract}
Given a quasi-median graph $X$, the crossing graph $\Delta X$ and the contact graph $\Gamma X$ are natural hyperbolic models of $X$. In this article, we show that the asymptotic translation length in $\Delta X$ or $\Gamma X$ of an isometry of $X$ is always rational. Moreover, if $X$ is hyperbolic, these rational numbers can be written with denominators bounded above uniformly; this is not true in full generality. Finally, we show that, if the quasi-median graph $X$ is constructible in some sense, then there exists an algorithm computing the translation length of every computable isometry. Our results encompass contact graphs in CAT(0) cube complexes and extension graphs of right-angled Artin groups. 
\end{abstract}

\tableofcontents

\section{Introduction}

Given a group $G$ and a finite generating set $S$, the \emph{translation number} of an element $g \in G$ is the limit
$$\tau_S(g):= \lim\limits_{n \to + \infty} \frac{\|g\|_S}{n},$$
where $\| \cdot\|_S$ denotes the word length associated to $S$. The limit exists by subadditivity, and it depends heavily on the chosen generating set. Introduced in \cite{MR1114609} with the purpose of studying subgroups in biautomatic groups, the structure of the \emph{translation spectrum} $\mathrm{TSpec}_S(G):= \{ \tau_S(g) \mid g \in G \text{ infinite order} \}$ have been investigated for many families of groups, including hyperbolic and CAT(0) groups \cite{MR919829, MR1334214, MR1390660, MR1741488, MR1776764, MR1694875}, solvable groups \cite{MR1466819, MR1612184, MR1736519}, small cancellation groups \cite{MR1357396, MR3950644}, Artin groups \cite{MR1714913}, mapping class groups \cite{MR1813237}, outer automorphism groups of free groups \cite{MR1934012}, Garside groups \cite{MR2303195, MR2344226, MR2438075}, and Coxeter groups \cite{MR2804516}. 

\medskip \noindent
Some of these results can be better understood in a more geometric framework. Namely, given a metric space $(X,d)$ and an isometry $g$, the (\emph{asymptotic} or \emph{stable}) \emph{translation length} of $g$ is the limit
$$\tau(g):= \lim\limits_{n \to + \infty} \frac{d(o, g^n \cdot o)}{n},$$
where $o \in X$ is a basepoint. The limit exists (by subadditivity) and does not depend on the chosen basepoint (by triangular inequality). Given a finitely generated group acting on one of its Cayley graphs by left-multiplication, the translation length of an element coincides with the corresponding translation number. Now, the question becomes: given our metric space $X$, what is the structure of the spectrum $\mathrm{TSpect}(X):= \{ \tau(g) \mid g \in \mathrm{Isom}(X) \text{ unbounded orbits}\}$? 

\medskip \noindent
In several families of graphs, there exists a dichotomy in the possible behaviour of an isometry: either it has bounded orbits or one of its powers acts as a translation on a bi-infinite geodesic. One consequence is that the translation spectrum is contained in the rational numbers $\mathbb{Q}$, or even in $\frac{1}{N} \mathbb{Z}$ for some integer $N \geq 1$ if one has a uniform control on the powers previously mentioned. Examples include locally finite hyperbolic graphs \cite{MR919829, MR1334214, MR1390660, MR1741488}; median graphs, or equivalently one-skeleta of CAT(0) cube complexes \cite{HaglundAxis}; quasi-median graphs with finite cliques (see Proposition~\ref{prop:Axis} below); bridged graphs, or equivalently one-skeleta of systolic complexes \cite{MR2507689}; and Helly graphs of finite combinatorial dimension \cite{HaettelOsajda}. 

\medskip \noindent
Instead of considering all the isometries of a metric space, one can focus on a specific subgroup of isometries. One case of interest, which received a lot of attention, is given by mapping class groups acting on their curve graphs. Among the results available in the vast literature on the subject, let us mention that: having a positive translation length characterises pseudo-Anosov elements \cite{MR1714338}; the translation length is always a rational number, with a uniform control on the denominator \cite{MR2367021, MR3378831}; and translation lengths can be computed algorithmically \cite{MR2705485, MR2970053, MR3378831}. 

\medskip \noindent
In order to emphasize the similarity between the \emph{extension graph} introduced in \cite{MR3039768} for right-angled Artin groups and the curve graph for mapping class groups, already motivated in \cite{MR3192368}, the recent work \cite{Rational} initiates the study of translation spectra of right-angled Artin groups acting on their extension graphs. Their main result shows that, given a finite connected graph $\Gamma$, the translation length of an element of the right-angled Artin group $A(\Gamma)$ on the extension graph $\Gamma^e$ is always rational. Moreover, if $\Gamma$ has girth $\geq 6$, then there is a uniform bound on the denominator. 

\medskip \noindent
In this article, our goal is to propose a more geometric and more general perspective on this result. 

\medskip \noindent
As noticed in \cite{QM}, the Cayley graph $\mathrm{QM}(\Gamma):=\mathrm{Cayl}(A(\Gamma), \bigcup_{v \in \Gamma} \langle v \rangle)$ of a right-angled Artin group $A(\Gamma)$ turns out to be a \emph{quasi-median graph}. As such, it geometry is encoded in the combinatorics of its \emph{hyperplanes}. Interestingly, the extension graph $\Gamma^e$, defined algebraically in \cite{MR3039768}, coincides with the \emph{crossing graph} of $\mathrm{QM}(\Gamma)$, namely the graph whose vertices are the hyperplanes of $\mathrm{QM}(\Gamma)$ and whose edges connect two hyperplanes whenever they are transverse. Thus, right-angled Artin groups acting on their extension graphs can be thought of as a particular case of the more general study of crossing graphs of quasi-median graphs.

\medskip \noindent
Crossing graphs have been initially introduced for median graphs (a.k.a.\ one-skeleta of CAT(0) cube complexes) independently in \cite{Roller, MR3217625}. A related graph introduced in \cite{MR3217625} is the \emph{contact graph}. Given a (quasi-)median graph $X$, its contact graph is defined as the graph whose vertices are the hyperplanes of $X$ and whose edges connect two hyperplanes whenever they are \emph{in contact} (i.e.\ transverse or tangent). In the article, we denote by $\Delta X$ the crossing graph and by $\Gamma X$ the contact graph.

\medskip \noindent
Quasi-median graphs and their crossing and contact graphs provide our general geometric framework. Thus, the main question we are interesting in is: given an isometry of a quasi-median graph, what can be said about its translation length in the corresponding crossing and contact graphs? Our first main result is the following Axis Theorem:

\begin{thm}\label{thm:IntroAxis}
Let $X$ be a quasi-median graph, and let $\Omega X$ be the crossing or contact graph of $X$. In the former case, we assume that $X$ has no cut-vertex; and, in any case, we assume that a vertex of $X$ belongs to $\leq N$ cliques for some fixed $N \geq 1$. For every $g \in \mathrm{Isom}(X)$ with unbounded orbits in $\Omega X$, there exists some $k \geq 1$ such that $g^k$ admits an axis, i.e.\ it acts as a translation on some bi-infinite geodesic in $\Omega X$. 
\end{thm}

\noindent
As already mentioned, such a dichotomy implies immediately that translation lengths must be rational numbers. 

\begin{cor}\label{cor:IntroRational}
Let $X$ be a quasi-median graph, and let $\Omega X$ be the crossing or contact graph of $X$. In the former case, we assume that $X$ has no cut-vertex; and, in any case, we assume that a vertex of $X$ belongs to $\leq N$ cliques for some fixed $N \geq 1$. For every $g \in \mathrm{Isom}(X)$, the translation length of $g$ in $\Omega X$ is rational. 
\end{cor}

\noindent
Regarding Theorem~\ref{thm:IntroAxis}, it is natural to ask whether there is a uniform upper bound on powers we can take, which would imply that the rational numbers from Corollary~\ref{cor:IntroRational} can be written with uniformly bounded denominators. This turns out to be true when the underlying quasi-median graph is hyperbolic, see Corollary~\ref{cor:HyperbolicCase}. However, this is not true in full generality. In Section~\ref{section:Examples}, we show that \cite{NoFactorSystem} provides the example of a group acting geometrically on a median graph that contains elements with arbitrarily small translation lengths in the crossing graph. 

\medskip \noindent
As another natural question, are the translation lengths in crossing and contact graphs algorithmically computable? In order to answer this question, we introduce \emph{constructible} quasi-median graphs and \emph{computable} isometries (see Section~\ref{section:Constructible}), and we prove the second main result of this article:

\begin{thm}\label{thm:AsymptoticLength}
Let $X$ be a constructible quasi-median graph. There exists an algorithm that computes, given a computable isometry $g \in \mathrm{Isom}(X)$, the asymptotic translation length of $g$ in $\Omega X$. 
\end{thm}

\noindent
Theorems~\ref{thm:IntroAxis} and~\ref{thm:AsymptoticLength} are already new for median graphs (a.k.a. one-skeleta of CAT(0) cube complexes), providing a vast range of applications. As already mentioned, including quasi-median graphs also allows to deduce information about right-angled Artin groups and their extension graphs. In fact, it is more natural to deal with \emph{graph products of groups}. Given a graph $\Gamma$ and a collection of groups $\mathcal{G}:= \{G_u \mid u \in V(\Gamma)\}$ indexed by its vertex-set $V(\Gamma)$, the graph product $\Gamma \mathcal{G}$ is
$$\langle G_u \ (u \in V(\Gamma)) \mid [G_u,G_v]=1 \ (\{u,v\} \in E(\Gamma)) \rangle$$
where $E(\Gamma)$ denotes the edge-set of $\Gamma$ and where $[G_u,G_v]=1$ is a shorthand for $[g,h]=1$ for all $g \in G_u$, $h \in G_v$. Usually, one says that graph products interpolate between free products (when $\Gamma$ has no edge) and direct sums (when $\Gamma$ is a complete graph). For instance, right-angled Artin groups coincide with graph products of infinite cyclic groups and right-angled Coxeter groups coincide with graph products of cyclic groups of order two. As shown in \cite{QM}, the Cayley graph
$$\mathrm{QM}(\Gamma, \mathcal{G}):= \mathrm{Cayl} \left( \Gamma \mathcal{G}, \bigcup\limits_{u \in V(\Gamma)} G_u \right)$$
is a median graph. We denote by $\Delta (\Gamma, \mathcal{G})$ its crossing graph. This is the natural generalisation of extension graphs for right-angled Artin groups. Algebraically, it can be defined as the graph whose vertices are the conjugates of vertex-groups and whose edges connect two subgroups whenever they commute (in the sense that every element of one subgroup commutes with every element of the other). 

\medskip \noindent
As an application of our general study of crossing graphs in quasi-median graphs, we get the following statement:

\begin{thm}\label{thm:IntroGP}
Let $\Gamma$ be a finite connected graph and $\mathcal{G}$ a collection of groups indexed by $V(\Gamma)$. The following statements hold.
\begin{itemize}
	\item For every $g \in \Gamma \mathcal{G}$, the translation length of $g$ in $\Delta(\Gamma, \mathcal{G})$ is rational.
	\item Moreover, if $\Gamma$ has no induced $4$-cycle, then this rational number can be written with a denominator $\leq |V(\Gamma)|^{40 \cdot \mathrm{clique}(\Gamma)}$.
	\item If the groups in $\mathcal{G}$ have solvable word problems, then there exists an algorithm that computes, given a $g \in \Gamma \mathcal{G}$, the translation length of $g$ in $\Delta(\Gamma, \mathcal{G})$. 
\end{itemize}
\end{thm}

\noindent
Theorem~\ref{thm:IntroGP} improves the main result of \cite{Rational} in several ways. First, it does not only apply to right-angled Artin groups but to arbitrary graph products, including for instance right-angled Coxeter groups. Next, even when restricted to right-angled Artin groups, Theorem~\ref{thm:IntroGP} weakens the conditions required to control the denominators of the rational translation lengths: we replace the condition of having girth $\geq 6$ with the condition of having no induced $4$-cycle. Finally, the (explicit) algorithm we describe offers the possibility to investigate computably the structure of translation spectra of right-angled Artin groups acting on their extension graphs, which remains poorly understood. For instance, we do not if such a spectrum can contain a value less than $2$, or when only integer values are taken. 

\medskip \noindent
However, it is worth mentioning that the upper bound given by the second item of Theorem~\ref{thm:IntroGP} is far from being optimal in general (which is not surprising since it is obtained from an argument applying to arbitrary quasi-median graphs). For instance, it is proved in \cite{Rational} that the upper bound for a right-angled Artin group defined by a graph $\Gamma$ of girth $\geq 6$ can be taken as the maximal degree of a vertex in $\Gamma$, which is optimal in some cases. This upper bound can be reproved geometrically in our framework, but the global picture remains unclear.

\paragraph{Acknowledgements.} I am grateful to Hyungryul Baik and Donggyun Seo for interesting discussions regarding their work \cite{Rational}, and to Sam Shepherd for interesting discussions about his recent preprint \cite{NoFactorSystem}.

\section{Quasi-median geometry}

\subsection{Generalities}

\noindent
There exist several equivalent definitions of quasi-median graphs; see for instance \cite{quasimedian}. Below is the definition used in \cite{QM}.

\begin{definition}
A connected graph $X$ is \emph{quasi-median} if it does not contain $K_4^-$ and $K_{3,2}$ as induced subgraphs, and if it satisfies the following two conditions:
\begin{description}
	\item[(triangle condition)] for every vertices $a, x,y \in X$, if $x$ and $y$ are adjacent and if $d(a,x)=d(a,y)$, then there exists a vertex $z \in X$ which adjacent to both $x$ and $y$ and which satisfies $d(a,z)=d(a,x)-1$;
	\item[(quadrangle condition)] for every vertices $a,x,y,z \in X$, if $z$ is adjacent to both $x$ and $y$ and if $d(a,x)=d(a,y)=d(a,z)-1$, then there exists a vertex $w \in X$ which adjacent to both $x$ and $y$ and which satisfies $d(a,w)=d(a,z)-2$.
\end{description}
\end{definition}

\noindent
The graph $K_{3,2}$ is the bipartite complete graph, corresponding to two squares glued along two adjacent edges; and $K_4^-$ is the complete graph on four vertices minus an edge, corresponding to two triangles glued along an edge. The triangle and quadrangle conditions are illustrated by Figure \ref{Quadrangle}.
\begin{figure}[h!]
\begin{center}
\includegraphics[scale=0.45]{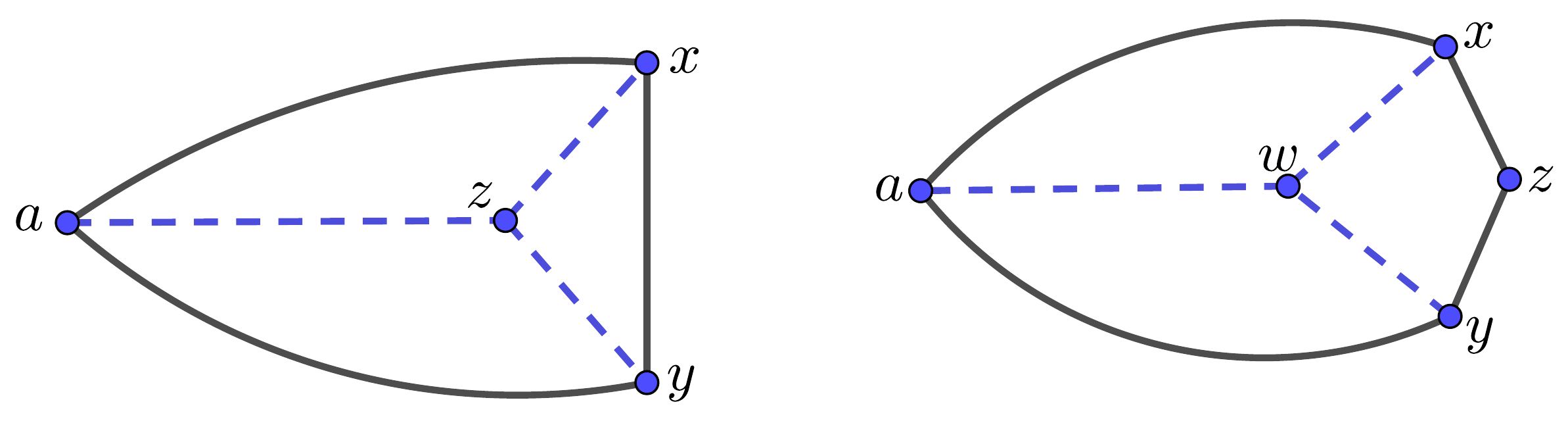}
\caption{Triangle and quadrangle conditions.}
\label{Quadrangle}
\end{center}
\end{figure}

\noindent
Recall that a \emph{clique} is a maximal complete subgraph, and that cliques in quasi-median graphs are gated \cite{quasimedian}. The fact quasi-median graphs do not contain induced copy of $K_4^-$ implies that the intersection between two distinct cliques is always either empty or reduced to a single vertex; in particular, an edge belongs to a unique clique. A \emph{prism} is a subgraph which a product of cliques. In the same way that median graphs can be naturally thought of as made of cubes, quasi-median graphs can be thought of as made of prisms. In a quasi-median graph, the maximal number of factors of a prism is referred to as its \emph{cubical dimension}. 

\paragraph{Hyperplanes.} A fundamental tool in the study of quasi-median graphs is given by \emph{hyperplanes}. 

\begin{definition}
Let $X$ be a graph. A \emph{hyperplane} $J$ is an equivalence class of edges with respect to the transitive closure of the relation saying that two edges are equivalent whenever they belong to a common triangle or are opposite sides of a square. We denote by $X \backslash \backslash J$ the graph obtained from $X$ by removing the interiors of all the edges of $J$. A connected component of $X \backslash \backslash J$ is a \emph{sector}. The \emph{carrier} of $J$, denoted by $N(J)$, is the subgraph generated by all the edges of $J$. Two hyperplanes $J_1$ and $J_2$ are \emph{transverse} if there exist two edges $e_1 \subset J_1$ and $e_2 \subset J_2$ spanning a $4$-cycle in $X$; and they are \emph{tangent} if they are not transverse but $N(J_1) \cap N(J_2) \neq \emptyset$. 
\end{definition}
\begin{figure}
\begin{center}
\includegraphics[trim={0 16.5cm 10cm 0},clip,scale=0.45]{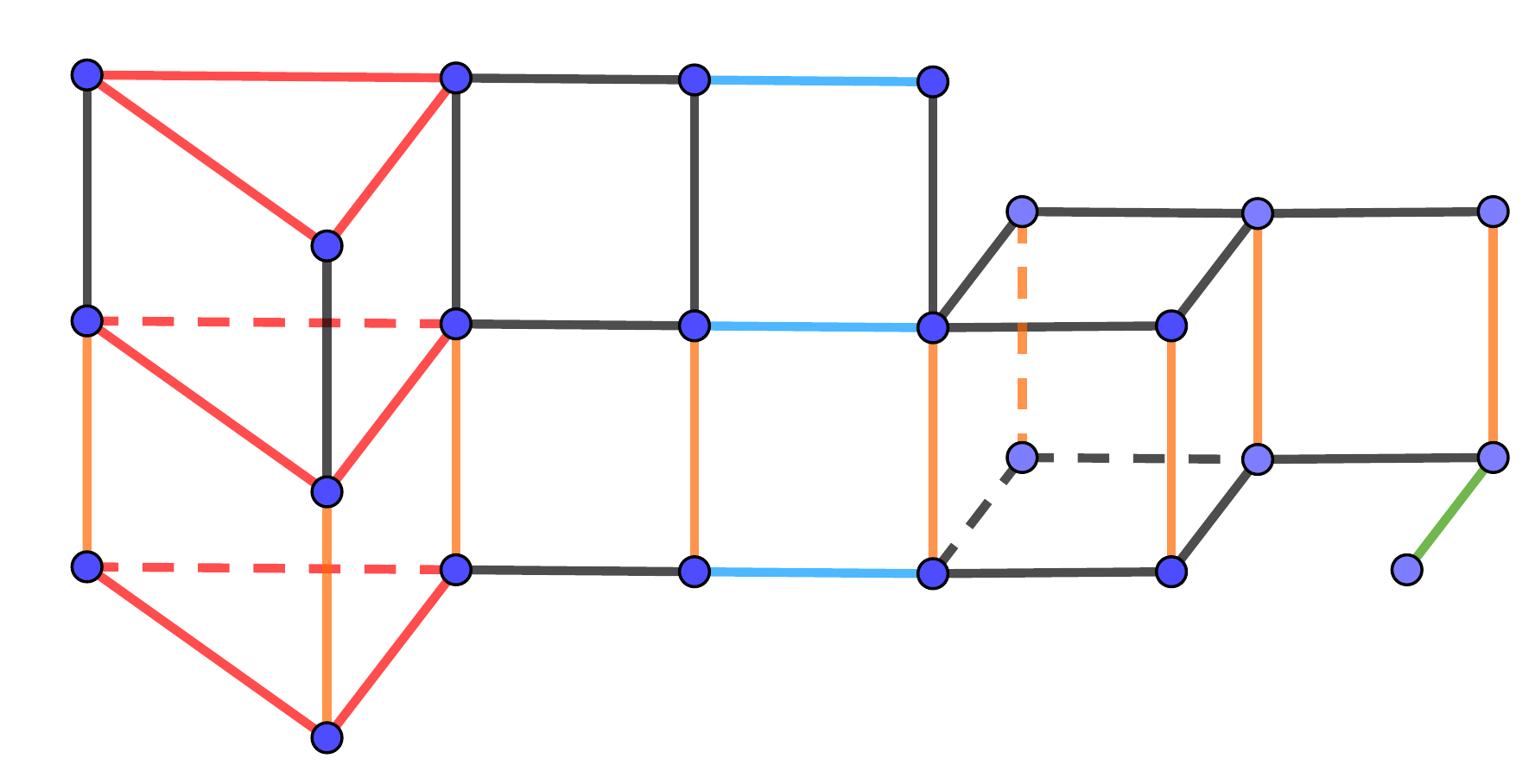}
\caption{A quasi-median graph and some of its hyperplanes.}
\label{figure3}
\end{center}
\end{figure}

\noindent
See Figure \ref{figure3} for examples of hyperplanes in a quasi-median graph.

\medskip \noindent
The key point is that the geometry of a quasi-median graph reduces to the combinatorics of its hyperplanes. This idea is motivated by the following statement:

\begin{thm}[\cite{QM}]\label{thm:BigQM}
Let $X$ be a quasi-median graph. 
\begin{itemize}
	\item[(i)] Every hyperplane $J$ separates $X$, i.e.\ $X \backslash \backslash J$ contains at least two connected components.
	\item[(ii)] Carriers and sectors are gated subgraphs.
	\item[(iii)] A path in $X$ is geodesic if and only if it intersects every hyperplane at most once.
	\item[(iv)] The distance between two vertices coincides with the number of vertices separating them.
\end{itemize}
\end{thm}

\paragraph{Projections.} As usual, a subgraph is \emph{convex} if it contains every geodesic between any two of its vertices. A strong convexity condition, \emph{gatedness}, is also very useful in quasi-median graphs:

\begin{definition}
Let $X$ be a graph and $Y \subset X$ a subgraph. A vertex $y \in Y$ is a \emph{gate} of an other vertex $x \in X$ if, for every $z \in Y$, there exists a geodesic between $x$ and $z$ passing through $y$. If every vertex of $X$ admits a gate in $Y$, then $Y$ is \emph{gated}.
\end{definition}

\noindent
It is worth noticing that the gate of $x$ in $Y$, when it exists, is unique and minimises the distance to $x$ in $Y$. As a consequence, it may be referred to as the \emph{projection} of $x$ onto $Y$. Also, gated subgraphs are automatically convex. Because an intersection of gated subgraphs is gated, we can define the \emph{gated hull} of a subset as the intersection of all the gated subgraphs containing in.

\medskip \noindent
Our next result, which can be found in \cite[Lemma~2.34]{QM}, describes how projections and hyperplanes interact.

\begin{prop}\label{prop:Projection}
Let $X$ be a quasi-median graph, $x \in X$ a vertex, and $Y \subset X$ a gated subgraph. Every hyperplane separating $x$ from its projection on $Y$ separates $x$ from $Y$. 
\end{prop}

\noindent
Let us record the following consequence, which will be useful later:

\begin{lemma}\label{lem:MinDistance}
Let $X$ be a quasi-median graph and $Y,Z \subset X$ two gated subgraphs. For every $y \in Y$, $\mathrm{proj}_Z(y)$ and $\mathrm{proj}_Y(\mathrm{proj}_Z(y))$ minimise the distance between $Y$ and $Z$. Moreover, a hyperplane separates these two vertices if and only if separates $Y$ and $Z$.
\end{lemma}

\begin{proof}
Let $p$ denote the projection of $y$ on $Z$, and $q$ the projection of $p$ on $Z$. According to Proposition~\ref{prop:Projection}, a hyperplane $J$ separating $p$ and $q$ separates $p$ from $Y$. A fortiori, it separates $y$ and $p$. Applying Proposition~\ref{prop:Projection} again, it follows that $J$ also separates $y$ from $Z$. Thus, every hyperplane separating $p$ and $q$ separates $Y$ and $Z$. Conversely, every hyperplane separating $Y$ and $Z$ has to separate $p$ and $q$. Clearly, the distance between $Y$ and $Z$ is bounded below by the number of hyperplanes separating them, so $d(p,q)$ must be equal to $d(Y,Z)$. 
\end{proof}

\paragraph{Median triangles.} Given a graph $X$ and a triple of vertices $(x_1,x_2,x_3)$, the triple $(a_1,a_2,a_3)$ is a \emph{median triangle} if it satisfies
$$d(x_i,x_j)=d(x_i,a_i)+d(a_i,a_j)+d(a_j,x_j) \text{ for all } i \neq j$$
and if it minises the quantity $d(a_1,a_2)+d(a_2,a_3)+d(a_1,a_3)$ under this condition. In median graphs, median triangles coincide with median points. 

\begin{prop}\label{prop:MedianTriangle}
In a quasi-median graph, every triple of vertices admits one and only one median triangle. Moreover, a hyperplane crossing this median triangle pairwise separates its vertices, which implies that the gated hull of the median triangle is a prism. 
\end{prop}

\noindent
See \cite[Proposition~2.84 and Fact~2.90]{QM}.

\subsection{Convex subgraphs}

\noindent
Recall that a subgroup is \emph{convex} if it contains all the geodesics between any two of its vertices. Because an intersection of convex subgraphs is again convex, we can define the \emph{convex hull} of a subset as the intersection of all the convex subgraphs containing it. The convex hull of two vertices $a$ and $b$ is referred to as the \emph{interval between $a$ and $b$}, denoted by $I(a,b)$. Alternatively, it is the union of all the geodesics connecting $a$ and $b$. According to \cite[Corollary~2.107]{QM}, the convex hull of a finite subset is always finite, so as a particular case:

\begin{lemma}\label{lem:IntervalFinite}
In a quasi-median graph, the interval between two vertices is always finite. 
\end{lemma}

\noindent
In a quasi-median graph, a \emph{multisector} is the subgraph induced by a union of sectors delimited by a given hyperplanes. For instance, the complement of a single sector, referred to as a \emph{cosector}, is a multisector. According to \cite[Proposition~2.104]{QM}, convex hulls can be characterised in terms of multisectors:

\begin{prop}\label{prop:ConvexHull}
Let $X$ be a quasi-median graph. The convex hull of a subset $S \subset X$ coincides with the intersection of all the multisectors containing $S$.
\end{prop}

\begin{cor}\label{cor:ConvexHull}
Let $X$ be a quasi-median graph and $S \subset X$ a set of vertices. Every sector intersecting the convex hull of $S$ intersects $S$ itself.
\end{cor}

\begin{proof}
If a sector does not intersect $S$, then its complement contains the convex hull of $S$ according to Proposition~\ref{prop:ConvexHull}, proving that our sector actually did not intersect the convex hull of $S$.
\end{proof}

\noindent
It follows from Lemma~\ref{lem:MinDistance} that, in a quasi-median graph, two disjoint gated subgraphs are separated by at least one hyperplane, in the sense that the two subgraphs lie into two distinct sectors. This is no longer true for convex subgraphs, but a weaker version of this property still holds. One says that two subsets in a quasi-median graph are \emph{weakly separated} by a hyperplane if they lie in two disjoint multisectors. Then:

\begin{prop}\label{prop:separation}
Let $X$ be a quasi-median graph and $A,B \subset X$ two convex subgraphs. If $A \cap B = \emptyset$ then there exists a hyperplane weakly separating $A$ and $B$. 
\end{prop}

\noindent
As a first observation towards the proof of the proposition, notice that, even though gated subgraphs are known to satisfy the Helly property (see for instance \cite[Propositon~2.8]{QM}), it clearly fails for convex subgraphs. (For instance, the three edges of a $3$-cycle pairwise intersect but they do not globally intersect.) Nevertheless:

\begin{lemma}\label{lem:HellyBis}
Let $X$ be a quasi-median graph, $Y \subset X$ a convex subgraph, and $Z_1,\ldots, Z_n \subset X$ gated subgraphs. If $Y,Z_1, \ldots, Z_n$ pairwise intersect, then their total intersect is non-empty.
\end{lemma}

\begin{proof}
It suffices to prove the lemma for $n=2$, the general case following easily by induction or by the Helly property for gated subgraphs. Fix three vertices $a \in Y \cap Z_1$, $b \in Y \cap Z_2$, and $c \in Z_1 \cap Z_2$. Let $(x,y,z)$ be the median triangle of $(a,b,c)$. Because $Z_1$ is gated, we have $a,x,y, z,c \in Z_1$; because $Z_2$ is gated, we have $b,x, y,z,c \in Z_2$; and because $Y$ is convex, we have $a,x,z,c \in Y$. Thus, $x,y \in Y \cap Z_1 \cap Z_2$ proving that $Y \cap Z_1 \cap Z_2$ is non-empty. 
\end{proof}

\noindent
Our next preliminary lemma is a weakened version of Proposition~\ref{prop:Projection} satisfied by convex subgraphs. 

\begin{lemma}\label{lem:projConv}
Let $X$ be a quasi-median graph, $Y \subset X$ a convex subgraph, $x \in X$ a vertex, and $M \subset Y$ the set of the vertices of $Y$ minimising the distance to $x$. For every $z \in M$, a hyperplane separating $x$ from $z$ either separates $x$ from $Y$ or crosses $M$. Moreover, any two hyperplanes crossing $M$ are transverse.
\end{lemma}

\begin{proof}
Let $J$ be a hyperplane separating $x$ from some $z \in M$. If $J$ does not cross $Y$, there is nothing to prove, so we assume that $J$ does cross $Y$. Consequently, there exists some $y \in Y$ such that $J$ separates $z$ and $y$. Let $(a,b,c)$ be the median triangle of $(x,y,z)$. Because $Y$ is convex, we have $y,b,c,z \in Y$. Necessarily $z=c$. Because $J$ separates both $z$ from $y$ and $z$ from $x$, necessarily $J$ crosses the median triangle. A fortiori, it separates $c=z$ and $b$. But $d(x,b)=d(x,c)=d(x,Y)$, so $b$ belongs to $M$. We conclude that $J$ crosses $M$, as desired.

\medskip \noindent
Let $J$ and $H$ be any two hyperplanes crossing $M$. Fix two vertices $y,z \in M$ separated by both $J$ and $H$. Let $(a,b,c)$ be the median triangle of $(x,y,z)$. Because $Y$ is convex, we have $y,b,c,z \in Y$. Necessarily, $y=b$ and $z=c$. Thus, $J$ and $H$ cross the median triangle, which implies that they must be transverse according to Proposition~\ref{prop:MedianTriangle}.
\end{proof}

\begin{proof}[Proof of Proposition~\ref{prop:separation}.]
If there exists a hyperplane (strongly) separating $A$ and $B$, then there is nothing to prove, so we assume that no hyperplane (strongly) separates $A$ and $B$. Fix two vertices $a \in A$ and $b \in B$ minimising the distance between $A$ and $B$, and let $Z_0$ denote the gated hull of $\{a,b\}$. The hyperplanes crossing $Z_0$ coincide with the hyperplanes separating $a$ and $b$, say $J_1, \ldots, J_n$. For every $1 \leq i \leq n$, assume that $J_i$ delimits some sector $Z_i$ that intersects both $A$ and $B$. 

\medskip \noindent
Observe that $Z_1, \ldots, Z_n$ pairwise intersect. Indeed, it follows from Lemma~\ref{lem:projConv} that a hyperplane separating $a$ and $b$ belongs one of the following families:
\begin{itemize}
	\item the hyperplanes separating $a$ from $B$ and crossing the set $M$ of the vertices of $A$ minimising the distance to $b$;
	\item the hyperplanes separating $b$ from $A$ and crossing the set $N$ of the vertices of $B$ minimising the distance to $a$;
	\item the hyperplanes crossing both $M$ and $N$.
\end{itemize}
Thus, each $Z_i$ either contains $B$ and separates $M$; or contains $A$ and separates $N$; or separates both $M$ and $N$. Any two sectors of the first two types intersect every sector of any type. And, because we also know from Lemma~\ref{lem:projConv} that two hyperplanes crossing both $M$ or both $N$ are transverse, it follows that two sectors of the third type also intersect. 

\medskip \noindent
We are now in good position to apply Lemma~\ref{lem:HellyBis} and conclude that the intersections $A \cap Z_0 \cap Z_1 \cap \cdots \cap Z_n$ and $B \cap Z_0 \cap Z_1 \cap \cdots \cap Z_n$ are both non-empty. But, because $Z_1, \ldots, Z_n$ are sectors delimited by the hyperplanes $J_1,\ldots, J_n$, which exhaust all the hyperplanes of $Z_0$, the intersection $Z_0 \cap Z_1 \cap \cdots \cap Z_n$ must be reduced to a single vertex. This vertex must then belong to both $A$ and $B$, contradicting the assumption $A \cap B= \emptyset$.

\medskip \noindent
Thus, we have proved that there exists a hyperplane separating $a$ and $b$ that weakly separates $A$ and $B$. 
\end{proof}

\noindent
We conclude this subsection with a last easy observation, which we record for future use.

\begin{lemma}\label{lem:ProjConvexGated}
Let $X$ be a quasi-median graph, $Y \subset X$ a convex subgraph, and $Z \subset X$ a gated subgraph. If $Y \cap Z \neq \emptyset$, then the projection on $Z$ of a vertex in $Y$ belongs to $Y \cap Z$.
\end{lemma}

\begin{proof}
Fix two vertices $y \in Y$ and $p \in Y \cap Z$. Let $z$ denote the projection of $y$ on $Z$ and let $(ab,c)$ be the median triangle of $(y,z,p)$. Because $z$ has to belong to $I(y,p)$, necessarily $z=a \in I(y,p) \subset Y$. 
\end{proof}

\subsection{Loxodromic isometries}

\noindent
As a consequence of \cite{HaglundAxis}, an isometry of a median graph, as soon as it has unbounded orbits and has no power that is an inversion, admits an \emph{axis}, i.e.\ a bi-infinite geodesic on which it acts as a translation. This section is dedicated to an analogous statement for quasi-median graphs. In this broader context, an \emph{inversion} is an isometry that stabilises a hyperplane and permutes non-trivially its sectors.

\begin{prop}\label{prop:Axis}
Let $X$ be a quasi-median graph and $g\in \mathrm{Isom}(X)$ an isometry. If $\langle g \rangle$ acts on $X$ without inversions and with unbounded orbits, then $g$ admits an axis in $X$. More precisely, for every $x \in \mathrm{Min}(g)$ and for every geodesic $[x,gx]$ between $x$ and $gx$, the concatenation of the $g^k [x,gx]$, $k \in \mathbb{Z}$, defines an axis of $g$. 
\end{prop}

\noindent
Here, $\mathrm{Min}(g):= \{ x \in X \mid d(x,gx)= \min \{ d(z,gz) \mid z \in X\} \}$ is the \emph{minimising set} of $g$. An isometry admitting an axis is referred to as a \emph{loxodromic isometry}. In fact, the second assertion of our proposition follows from the first one according to our next general observation:

\begin{lemma}\label{lem:AxisFromAnother}
Let $X$ be a graph and $g \in \mathrm{Isom}(X)$ an isometry. Assume that there exists a bi-infinite geodesic $\gamma$ on which $g$ acts as a translation of length $\ell$. Then $d(x,gx) \leq \ell$ for every $x \in X$. Moreover, if $x \in X$ is vertex satisfying $d(x,gx)= \ell$, then $x$ belongs to an axis of $g$. 
\end{lemma}

\begin{proof}
Fix two vertices $x \in X$ and $y \in \gamma$. Then, for every $n \geq 1$, we have
$$n \ell - 2d(x,y) = d(y,g^ny)-2d(x,y) \leq d(x,g^nx) \leq n d(x,gx),$$
hence $\ell -2d(x,y)/n \leq d(x,gx)$. When $n\to + \infty$, we find $\ell \leq d(x,gx)$. This proves the first assertion of our lemma.

\medskip \noindent
Now, fix a vertex $x \in X$ satisfying $d(x,gx)=\ell$. In order to show that $x$ belongs to an axis of $g$, we need to prove that the $\langle g \rangle$-translates of $x$ all lie on a common bi-infinite geodesic. It suffices to show that $d(x,g^nx) = \sum_{i=0}^{n-1} d(g^ix,g^{i+1}x)$ for every $n \geq 1$. Observe that
$$n \ell \leq d(x,g^nx) \leq n d(x,gx)=n \ell,$$
where the first inequality is obtained by applying our previous assertion to $g^n$. So
$$d(x,g^nx) = n \ell = \sum\limits_{i=0}^{n-1} d(g^ix,g^{i+1} x),$$
as desired. 
\end{proof}

\noindent
Our strategy to prove Proposition~\ref{prop:Axis} differs from \cite{HaglundAxis} (thus providing an alternative proof of the main result of \cite{HaglundAxis}). The idea is the following. Starting from an arbitrary vertex $x$ in our quasi-median graph $X$, the orbit $\langle g \rangle \cdot x$ under our isometry $g$ has to stay at bounded distance from the axis we are looking for. Loosely speaking, $\langle g \rangle  \cdot x$ coincides with this axis up to some noise we have to remove. In order to smoothen the quasi-line $\langle g \rangle \cdot x$, we fix a large integer $N \geq 1$, define a reasonable centre $c_k$ of $\{g^{k+i}x \mid -N \leq k \leq N\}$ for every $k \in \mathbb{Z}$, and show that the $c_k$ all lie on a common bi-infinite geodesic. 

\begin{proof}[Proof of Proposition~\ref{prop:Axis}.]
Given a vertex $x \in X$, the action of $\langle g \rangle$ on the gated hull of the orbit $\langle g \rangle x$ has only finitely many orbits of hyperplanes. Indeed, every hyperplane separating two vertices in the orbit $\langle g \rangle x$ admits a $\langle g \rangle$-translate separating $x$ and $d(x,gx)$. Fix a vertex $x \in x$ such that the action of $\langle g \rangle$ on the gated hull of $\langle g \rangle x$ has the smallest possible number of orbits of hyperplanes.

\begin{claim}\label{claim:CalJEmpty}
For every hyperplane $J$, the set $\{ k \in \mathbb{Z} \mid \text{$J$ separates $g^kx$ and $g^{k+1}x$}\}$ is finite.
\end{claim}

\noindent
Let $\mathcal{J}$ denote the set of the hyperplanes $J$ separating $g^kx$ and $g^{k+1}x$ for infinitely many $k \in \mathbb{Z}$. 

\medskip \noindent
First, observe that $\mathcal{J}$ is finite. Indeed, it is clear that $\mathcal{J}$ is $\langle g \rangle$-invariant. And, because every hyperplane in $\mathcal{J}$ has a $\langle g \rangle$-translate separating $x$ and $gx$ while exactly $d(x,hx)$ hyperplanes separate $x$ and $gx$, $\mathcal{J}$ contains only finitely many $\langle g \rangle$-orbits and they all have sizes $\leq d(x,gx)$. 

\medskip \noindent
Fix a hyperplane $J \in \mathcal{J}$ and a sector $S$ delimited by $J$. If the $g^kS$, $k \in \mathbb{Z}$, pairwise intersect, then $\bigcap_{k \in \mathbb{Z}} g^kS$ yields a $\langle g \rangle$-invariant gated subgraph not crossed by any translate of $J$. This contradicts our choice of $x$. So there exists some $h_1 \in \langle g \rangle$ such that $h_1S \cap S = \emptyset$. Of course, $h_1J$ and $J$ cannot be transverse, but they also have to be distinct because $h_1$ is not an inversion. Let $S'$ denote the sector delimited by $J$ and containing $h_1J$. Because $\mathcal{J}$ is $\langle g \rangle$-invariant and finite, we cannot have $h_1S' \subsetneq S'$, so $h_1S'$ must be the sector delimited by $h_1J$ containing $J$. As before, if the $g^kS$, $k \in \mathbb{Z}$, pairwise intersect, we get a contradiction; so there must exist some $h_2 \in \langle g \rangle$ satisfying $h_2 S' \cap S' = \emptyset$. But then $h_2S' \subset (S')^c \subsetneq h_1S'$, hence $h_1^{-1}h_2 S' \subsetneq S'$. Again, this contradicts the fact that $\mathcal{J}$ is finite.

\medskip \noindent
Thus, the only possibility is that $\mathcal{J}$ is empty, concluding the proof of Claim~\ref{claim:CalJEmpty}. 

\medskip \noindent
From now on, we fix a geodesic $[x,gx]$ between $x$ and $gx$, and we denote by $\gamma$ the concatenation of the $\langle g \rangle$-translates of $[x,gx]$. We enumerate the edges of $\gamma$ as $\ldots, e_{-1},e_0, e_1, \ldots$ such that $e_i$ and $e_{i+1}$ have a common endpoint for every $i \in \mathbb{Z}$. Claim~\ref{claim:CalJEmpty} allows us to define, for every hyperplane $J$ crossing $\langle g \rangle x$, the numbers
$$\ell(J) := \min \{ k \in \mathbb{Z} \mid e_k \in J \} \text{ and }  r(J):= \max \{ k \in \mathbb{Z} \mid e_k \in J \} .$$
In other words, $\ell(J)$ (resp. $r(J)$) indicates the leftmost (resp. rightmost) part of $\langle g \rangle x$ crossed by $J$. Notice that $\ell(gJ)=\ell(J)+1$ and $r(gJ)=r(J)+1$. As a consequence, the difference $r(J) - \ell(J)$ depends only on the $\langle g \rangle$-orbit of $J$. Because there are only finitely many orbits of hyperplanes crossing $\langle g \rangle x$, we can fix an integer $N \geq 1$ larger than any of these quantities.

\medskip \noindent
We distinguish two types of hyperplanes crossing $\langle g \rangle x$. Such a hyperplane $J$ is \emph{one-sided} if all the $e_k$ for $k \leq \ell(J)$ and $k \geq r(J)$ belong to the same sector delimited by $J$, which we denote by $S(J)$. Otherwise, $J$ is \emph{two-sided} and we denote by $L(J)$ (resp. $R(J)$) the sector delimited by $J$ containing the $e_k$ for $k \leq \ell(J)$ (resp. $k \geq r(J)$). 

\medskip \noindent
For every $k \in \mathbb{Z}$, set $B_k:= \bigcup_{-N \leq i \leq N} e_{k+i}$. Notice that $gB_k=B_{k+1}$. Given a $k \in \mathbb{Z}$, we want to define a centre of $B_k$. For this purpose, given a hyperplane $J$, we define the sector $S_k(J)$ delimited by $J$ as follows:
\begin{itemize}
	\item if $J$ is one-sided, $S_k(J):=S(J)$;
	\item if $J$ is two-sided and $\ell(J) \geq k$, $S_k(J):=L(J)$;
	\item if $J$ is two-sided and $\ell(J)<k$, $S_k(J):=R(J)$.
\end{itemize}
Roughly speaking $S_k(J)$ represents the sector of $J$ that contains the bigger part of $B_k$. 

\begin{claim}\label{claim:PairwiseIntersect}
The intersection between the gated hull $\mathrm{GH}(B_k)$ of $B_k$ and the $S_k(J)$ for $J$ crossing $B_k$ is reduced to a single vertex, which we denote by $c_k$. 
\end{claim}

\noindent
It suffices to show that $\mathrm{GH}(B_k)$ and the $S_k(J)$ for $J$ crossing $B_k$ pairwise intersect. Then the Helly property for gated subgraphs assures that the total intersection is non-empty, and, because we are choosing one sector for each hyperplane crossing $B_k$, there cannot be two vertices in the intersection.

\medskip \noindent
If $J_1$ is a one-sided hyperplane, then $S(J_1)$ clearly intersects $S(J_2)$ if $J_2$ is another one-sided hyperplane and both $L(J_2)$ and $R(J_2)$ if $J_2$ is a two-sided hyperplane. Moreover, if $J_1$ and $J_2$ are two two-sided hyperplanes, then $R(J_1) \cap R(J_2)$ and $L(J_1) \cap L(J_2)$ are both non-empty. So, given two hyperplanes $J_1$ and $J_2$ crossing $B_k$, the only remaining possibility in order to have $S_k(J_1) \cap S_k(J_2) = \emptyset$ is that $J_1,J_2$ are two non-transverse two-sided hyperplanes with $S_k(J_1)=L(J_1)$ and $S_k(J_2)=R(J_2)$ (up to switching $J_1$ and $J_2$). Necessarily, from $S_k(J_1) \cap S_k(J_2)$, it follows that $\ell(J_2) < k \leq \ell(J_1)$; and because $J_1$ and $J_2$ are not transverse, we know that $r(J_2) < \ell(J_1)$. Therefore, $S_k(J_1)=L(J_1)$ and $S_k(J_2)=R(J_2)$ actually intersect.

\medskip \noindent
Next, if $J$ is a one-sided hyperplane crossing $B_k$, then $B_k$ cannot lie in the complement of $S(J)$ by definition of $N$, so $S_k(J) \cap \mathrm{GH}(B_k) \neq \emptyset$. If $J$ is two-sided with $L(J) \cap \mathrm{GH}(B_k)= \emptyset$, then $\{e_i \mid i \leq \ell(J) \} \cap B_k$ must be empty, which implies that $k-N > \ell(J)$. We deduce that $k> r(J)$, $g^kx \in R(J)$, and $S_k(J)=R(J)$. Thus, $e_k$ belongs to $S_k(J) \cap B_k$, proving that $S_k(J)$ and $\mathrm{GH}(B_k)$ intersect. This concludes the proof of Claim~\ref{claim:PairwiseIntersect}. 

\medskip \noindent
Notice that $gc_k=c_{k+1}$ for every $k \in \mathbb{Z}$. We claim that the $c_k$ all belong to a common bi-infinite geodesic. This proves that, fixing a geodesic $[c_0,c_1]$ between $c_0$ and $c_1=gc_0$, the concatenation of the $\langle g \rangle$-translates of $[c_0,c_1]$ defines an axis for $g$. In order to prove our claim, it suffices to observe that: $c_k$ belongs to $S(J)$ for every $k \in \mathbb{Z}$ if $J$ is a one-sided hyperplane; and, if $J$ is two-sided, $c_k$ belongs to $L(J)$ for $k \leq \ell(J)$ and to $R(J)$ for $k>\ell(J)$. Therefore, no hyperplane can separate a $c_k$ from some $c_i$ and $c_j$ with $i<k<j$. 

\medskip \noindent
So far, we have proved the first assertion of Proposition~\ref{prop:Axis}. The second assertion follows from Lemma~\ref{lem:AxisFromAnother}.
\end{proof}

\subsection{Quasiconvex isometries}

\noindent
A loxodromic isometry is \emph{quasiconvex} if it admits a quasiconvex axis. In order to quantify this property, we introduce the two following quantities:

\begin{definition}
Let $X$ be a quasi-median graph and $\gamma$ a bi-infinite geodesic. Define 
\begin{itemize}
	\item $\mathrm{QC}(\gamma)$ as the largest Hausdorff distance between two bi-infinite geodesics in the convex hull of $\gamma$. 
	\item $\mathrm{HQC}(\gamma)$ as the maximal $n \geq 0$ for which there exist two transverse collections of hyperplanes of size $n$ crossing~$\gamma$.
\end{itemize}
\end{definition}

\noindent
Thus, $\gamma$ is quasiconvex if and only if the quantity $\mathrm{QC}(\gamma)$ is finite. The quantity $\mathrm{HQC}(\gamma)$ is a more median-friendly version of $\mathrm{QC}(\gamma)$ which coarsely coincides with $\mathrm{QC}(\gamma)$ if $\gamma$ is an axis according to Lemma~\ref{lem:FormulaQC} below. Our first observation is that $\mathrm{HQC}(\cdot)$ does not depend on a particular choice of an axis. Indeed:

\begin{lemma}\label{lem:TwoAxes}
Let $X$ be a quasi-median graph and $g \in \mathrm{Isom}(X)$ a loxodromic isometry. Any two axes of $g$ cross exactly the same hyperplanes.
\end{lemma}

\begin{proof}
Let $\gamma_1,\gamma_2$ be two axes of $g$. First, we want to show that $\gamma_1$ and $\gamma_2$ fellow-travel in the following sense:

\begin{claim}\label{claim:AxesFellowTravel}
There exist $C,D \geq 0$ such that, for every $t_1 \in \mathbb{Z}$, there exists some $t_2 \in \mathbb{Z}$ such that $|t_1-t_2| \leq C$ and $d(\gamma_1(t_1),\gamma_2(t_2)) \leq D$.
\end{claim}

\noindent
Fix two $s_1,s_2 \in \mathbb{Z}$. Set $C:=\|g\|$ and $D:= \|g\|+ d(\gamma_1(s_1),\gamma_2(s_2))$. For every $t_1 \in \mathbb{Z}$, there exists some $r \in \mathbb{Z}$ such that $d(\gamma_1(t_1), g^r \gamma_1(s_1)) \leq \|g\|$. Set $t_2:= s_1+ r \|g\|$. The latter inequality implies that
$$|t_1-t_2| = d(\gamma_1(t_1), \gamma_1(s_1+ r \|g\|)) \leq \|g\| = C.$$
Moreover, we have
$$\begin{array}{lcl} d(\gamma_1(t_1),\gamma_2(t_2)) & \leq & d(\gamma_1(t_1), \gamma_1(s_1+r \|g\|)) + d(\gamma_1(s_1+r \|g\|), \gamma_2(t_2)) \\ \\ & \leq & d(\gamma_1(t_1), g^r \gamma_1(s_1)) + d(g^r \gamma_1(s_1), g^r \gamma_2(s_2)) \\ \\ & \leq & \|g\| + d(\gamma_1(s_1),\gamma_2(s_2)) = D \end{array}$$
which concludes the proof of Claim~\ref{claim:AxesFellowTravel}.

\medskip \noindent
Now, let $J$ be a hyperplane crossing $\gamma_1$. If $J$ does not cross $\gamma_2$, then there exists some $k_0 \in \mathbb{Z}$ such that $J$ separates $\gamma_1(k)$ from $\gamma_2$ for every $k \geq k_0$. Given an $s \geq 0$ larger than the Hausdorff distance between $\gamma_1$ and $\gamma_2$, which is finite as a consequence of Claim~\ref{claim:AxesFellowTravel}, we deduce that $J,gJ, \ldots, g^s J$ separate $\gamma_1(k+s \|g\|)$ from $\gamma_2$ for every $k \geq k_0$, which is impossible. Thus, every hyperplane crossing $\gamma_1$ has to cross $\gamma_1$ as well. Symmetrically, every hyperplane crossing $\gamma_2$ has to cross $\gamma_1$.
\end{proof}

\noindent
Thus, we can safely define:

\begin{definition}
Let $X$ be a quasi-median graph and $g \in \mathrm{Isom}(X)$ a loxodromic isometry. Define $\mathrm{HQC}(g)$ as $\mathrm{HQC}(\gamma)$ for an arbitrary axis $\gamma$ of $g$. 
\end{definition}

\noindent
We conclude this subsection by proving that $\mathrm{HQC}(g)$ coarsely coincides with $\mathcal{QC}(\gamma)$ for any choice of an axis $\gamma$ of $g$.

\begin{lemma}\label{lem:FormulaQC}
Let $X$ be a quasi-median graph of finite cubical dimension and $\gamma$ an axis of some isometry $g \in \mathrm{Isom}(X)$. The the inequality
$$\mathrm{HQC}(g) \leq \mathrm{QC}(\gamma) \leq 2 \cdot \mathrm{HQC}(g)$$
holds. 
\end{lemma}

\begin{proof}
For short, set $M:= \mathrm{HQC}(g)$. We begin by proving the following observation:

\begin{claim}\label{claim:ForQC}
Let $\alpha, \beta \subset \mathrm{QC}(\gamma)$ be two bi-infinite geodesics. Assume that every hyperplane crossing $\gamma$ also crosses $\beta$. Then $\alpha$ lies in the $2M$-neighbourhood of $\beta$. 
\end{claim}

\noindent
Given two vertices $p \in \alpha$ and $q \in \beta$, fix two vertices $q^-,q^+ \in \beta$ such that all the hyperplanes separating $p$ and $q$ (which must cross $\gamma$, and a fortiori $\beta$) cross $\beta$ between $q^-$ and $q^+$. Observe that $p$ belongs to $I(q^-,q^+)$, because a hyperplane separating $p$ from $\{q^-,q^+\}$ would have to cross $\beta$ elsewhere than between $q^-$ and $q^+$. 

\medskip \noindent
\begin{minipage}{0.48\linewidth}
\includegraphics[width=0.95\linewidth]{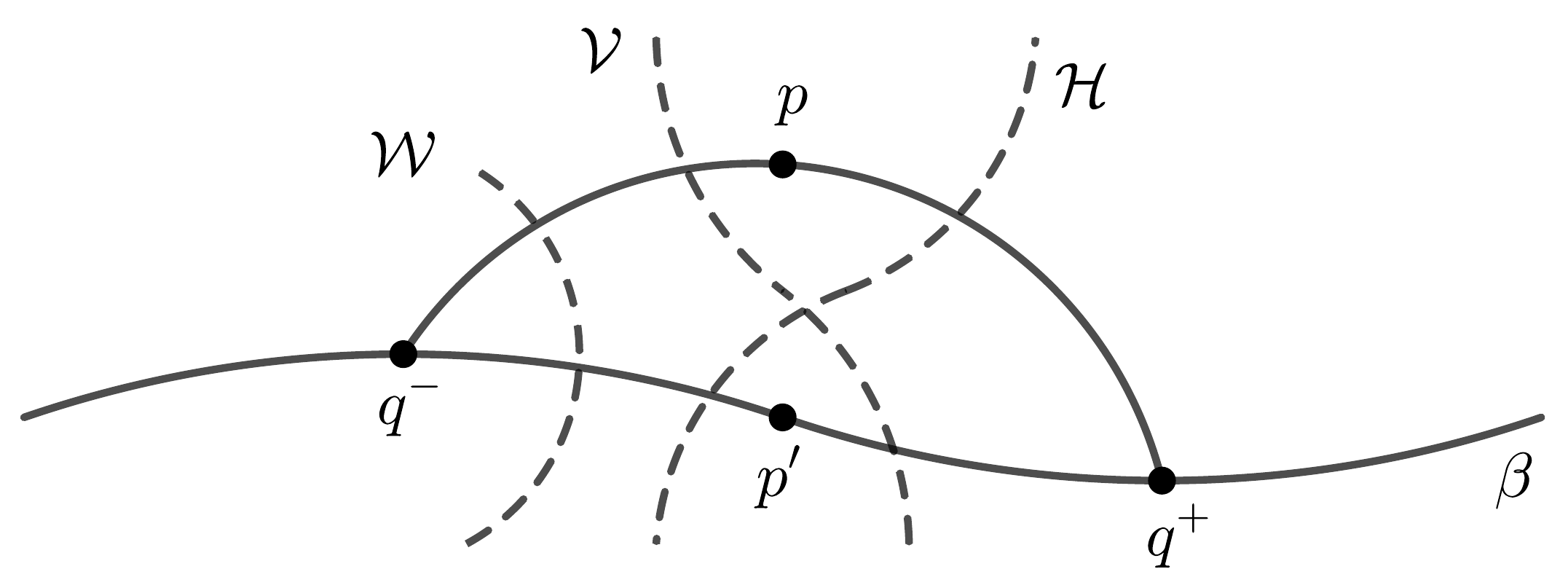}
\end{minipage}
\begin{minipage}{0.5\linewidth}
Let $p'$ denote the vertex of $\beta$ between $q^-$ and $q^+$ satisfying $d(q^-,p)=d(q^+,p')$. Let $\mathcal{H}$ denote the set of the hyperplanes separating $\{p,q^-\}$ and $\{p',q^+\}$, and $\mathcal{V}$ the set of the hyperplanes separating $\{q^-,p'\}$ and $\{p,q^+\}$.  
\end{minipage}

\medskip \noindent
If $\mathcal{W}$ denotes the set of the hyperplanes separating $q^-$ from $\{p,p',q^+\}$, then
$$|\mathcal{W}| + | \mathcal{H} | = d(q^-,p')=d(q^-,p) = |\mathcal{W}|+|\mathcal{V}|,$$
so $\mathcal{H}$ and $\mathcal{V}$ have the same size. Moreover, $\mathcal{H}$ and $\mathcal{V}$ are clearly transverse. Observe that $\mathcal{H} \cup \mathcal{V}$ coincides with the set of the hyperplanes separating $p$ and $p'$, hence
$$d(p, \beta) \leq d(p,p') = | \mathcal{H}| + |\mathcal{V}| \leq 2M.$$
Thus, we have proved that $\alpha$ lies in the $2M$-neighbourhood of $\beta$. Claim~\ref{claim:ForQC} is proved. 

\medskip \noindent
If $\mathrm{QC}(\gamma)$ is infinite, then applying Claim~\ref{claim:ForQC} to $\beta=\gamma$ and $\alpha$ arbitrary shows that $M$ must be infinite as well. A fortiori, the inequality $\mathrm{QC}(\gamma) \leq 2M$ holds.

\medskip \noindent
Now, assume that $\mathrm{QC}(\gamma)$ is finite, which amounts to saying that $\gamma$ is quasi-dense in its convex hull $\mathrm{CH}(\gamma)$. Given a bi-infinite geodesic $\zeta$ in $\mathrm{CH}(\gamma)$, we know from Corollary~\ref{cor:ConvexHull} that every hyperplane crossing $\zeta$ also crosses $\gamma$. It turns out that, because $\gamma$ is an axis, the converse also holds, i.e.\ every hyperplane crossing $\gamma$ also crosses $\zeta$. Indeed, if there exists some hyperplane $J$ crossing $\gamma$ but not $\zeta$, then $J$ separates $\zeta$ from some infinite subray $\gamma_0 \subset \gamma$. up to replacing $g$ with its inverse, we assume that $g\gamma_0 \subset \gamma_0$. Because $X$ has finite cubical dimension, there exists some $1 \leq d \leq \dim(X)$ such that $J$ and $g^dJ$ are not transverse. Then, for every $k \geq 1$, the hyperplanes $J, g^dJ, \ldots ,g^{kd}J$ separate $\zeta$ from some infinite subray of $\gamma_0$. Taking $k$ larger than the Hausdorff distance between $\gamma$ and $\zeta$ yields a contradiction. Thus, we have proved that every bi-infinite geodesic in $\mathrm{CH}(\gamma)$ is crossed exactly by the hyperplanes crossing $\gamma$. 

\medskip \noindent
Consequently, given two bi-infinite geodesics $\alpha, \beta$ in $\mathrm{CH}(\gamma)$, we deduce by applying Claim~\ref{claim:ForQC} twice that the Hausdorff distance between $\alpha$ and $\beta$ is at most $2M$. Therefore, $\mathrm{QC}(\gamma) \leq 2M$. 

\medskip \noindent
Finally, let $\mathcal{H}, \mathcal{V}$ be two finite transverse collections of hyperplanes of the same size crossing $\gamma$. Fix two vertices $a,b \in \gamma$ such that the hyperplanes in $\mathcal{H} \cup \mathcal{V}$ cross $\gamma$ between $a$ and $b$. For every $J \in \mathcal{H} \cup \mathcal{V}$, let $J^-$ (resp. $J^+$) denote the sector delimited by $J$ containing $a$ (resp. $b$). Fix two arbitrary vertices
$$p \in \bigcap\limits_{J \in \mathcal{H}} J^- \cap \bigcap\limits_{J \in \mathcal{V}} J^+=:A \text{ and } q \in \bigcap\limits_{J \in \mathcal{H}} J^+ \cap \bigcap\limits_{J \in \mathcal{V}} J^-=:B.$$
Let $(x,y,z)$ be the median triangle of $(a,b,p)$. We deduce from the convexity of sectors that $z$ belongs to $A$. For the same reason, we know that no hyperplane in $\mathcal{H}$ (resp. in $\mathcal{V}$) can separate $x$ and $z$ (resp. $y$ and $z$). But the hyperplanes separating $x$ and $z$ coincide with the hyperplanes separating $y$ and $z$ according to Proposition~\ref{prop:MedianTriangle}, so no hyperplane in $\mathcal{H} \cup \mathcal{V}$ can separate $x$ and $z$, which implies that $x \in I(a,b) \cap D$. Thus, we have proved that there exists a geodesic $\zeta_0$ between $a$ and $b$ passing through $A$. Let $\zeta$ denote the bi-infinite geodesic obtained from $\gamma$ by replacing the subsegment between $a$ and $b$ with $\zeta_0$. Similarly, there exists a geodesic $\xi_0$ between $a$ and $b$ passing through $B$. Let $\xi$ denote the bi-infinite geodesic obtained from $\gamma$ by replacing the subsegment between $a$ and $b$ with $\xi_0$. The key observation is that a vertex in $\zeta \cap A$ is separated from a given vertex in $\zeta$ either by all the hyperplanes in $\mathcal{H}$ or by all the hyperplanes in $\mathcal{V}$, so the Hausdorff distance between $\alpha$ and $\beta$ must be at least $|\mathcal{H}|=|\mathcal{V}|$. Thus, we have proved that $\mathrm{QC}(\gamma) \geq M$. 
\end{proof}

\begin{remark}
It can be proved that, given a loxodromic isometry $g$ of a quasi-median graph $X$, the quantity $\mathrm{QC}(\gamma)$ does not depend on the axis $\gamma$ we choose. More precisely, the convex hulls of any two axes turn out to be always isometric. However, this observation will be useful in the sequel as $\mathrm{HQC}(g)$ will be a quantity easier to compute. 
\end{remark}

\section{Crossing and contact graphs}

\noindent
Recall that, given a quasi-median graph $X$, the \emph{crossing graph} $\Delta X$ is the graph whose vertices are the hyperplanes of $X$ and whose edges connect two hyperplanes whenever they are transverse; and the \emph{contact graph} $\Gamma X$ is the graph whose vertices are the hyperplanes of $X$ and whose edges connect two hyperplanes whenever they are \emph{in contact} (i.e.\ when their carriers intersect). In this section, we record basic properties of crossing and contact graphs. These properties are well-known for median graphs, and can be found in \cite{MR3217625, MR4057355}; crossing graphs of some specific quasi-median graphs are also studied in \cite{QM, AutGP}. The arguments used in this section are straightforward adaptations.

\subsection{Comparison of the two graphs}

\noindent
Even though crossing and contact graphs have similar definitions, there are a few fundamental differences. The first one is that the contact graph is always connected when the crossing graph may be disconnected. However, the connectedness of the crossing graph can be detected by the existence of cut-vertices. (For median graphs, this observation can be found in \cite{MR1881707}.)

\begin{lemma}
The contact graph of a quasi-median graph is always connected.
\end{lemma}

\begin{proof}
Let $X$ be a quasi-median graph and $A,B$ two hyperplanes. Fix an arbitrary path from $N(A)$ to $N(B)$ and let $J_1, \ldots, J_n$ denote the hyperplanes successively crossed. Then the hyperplanes $A,J_1, \ldots, J_n,B$ are successively in contact, producing a path from $A$ to $B$ in the contact graph. 
\end{proof}

\noindent
Given a quasi-median graph $X$, its \emph{crossing graph} $\Delta X$ is the graph whose vertices are the hyperplanes of $X$ and whose edges link two hyperplanes whenever they are transverse. The crossing graph is not always connected, but we understand the default of connectedness:

\begin{lemma}\label{lem:CrossingConnected}
Let $X$ be a quasi-median graph. Its crossing graph is disconnected if and only if there exists a cut-vertex in $X$.
\end{lemma}

\begin{proof}
Given a vertex $x \in X$, define its clique-link as the graph whose vertices are the cliques containing $x$ and whose edges connect two vertices if the corresponding cliques span a prism (or equivalently, if the hyperplanes containing the cliques are transverse). We can use clique-links in order to characterise cut-vertices:

\begin{claim}\label{claim:CutVertex}
A vertex $x \in X$ is a cut-vertex if and only if its clique-link is disconnected.
\end{claim}

\noindent
First, assume that the clique-link of $x$ is connected. Given two vertices $a,b \in X$, fix a path between $a$ and $b$, say a geodesic $[a,b]$. If $x$ does not belong to $[a,b]$, then $x$ does not separate $a$ and $b$. Otherwise, let $e_1, \ldots, e_n$ be a sequence of edges having $x$ as an endpoint such that the cliques containing $e_1, \ldots, e_n$ defines a path in the clique-link of $x$, such that $e_1$ is the first edge of $[a,b]$ that belongs to a clique containing $x$, and such that $e_n$ is the last edge of $[a,b]$ that belongs to a clique containing $x$. For every $1 \leq i \leq n-1$, the edges $e_i$ and $e_{i+1}$ span a $4$-cycle, say $e_i \cup e_{i+1} \cup \epsilon_{i+1} \cup \epsilon_i$. Then $\epsilon_1, \ldots, \epsilon_n$ defines a path connecting the endpoints of $e_1$ and $e_n$ distinct from $x$ that does not pass through $x$. Thus, replacing the subsegment of $[a,b]$ between $e_1$ and $e_n$ (including $e_1$ and $e_n$ themselves) with $\epsilon_1 \cup \cdots \cup \epsilon_n$ produces a path between $a$ and $b$ that does not pass through $x$. This proves that $x$ is not a cut-vertex.

\medskip \noindent
Conversely, assume that the clique-link of $x$ is disconnected. Let $\epsilon_1,\epsilon_2$ be two edges containing $x$ that belong to two cliques lying in distinct connected components of the clique-link of $x$. Let $a_1,a_2$ denote the endpoints of $\epsilon_1,\epsilon_2$ distinct from $x$. We claim that $a_1$ and $a_2$ are separated by $x$. The key point is that any path between $a_1$ and $a_2$ can be obtained from $\epsilon_1 \cup \epsilon_2$ by some elementary operations. More precisely, given an oriented path $\gamma$ in our graph $X$, which we decompose as a concatenation of oriented edges $e_1 \cdots e_n$, one says that $\gamma'$ is obtained from $\gamma$ by
\begin{itemize}
	\item \emph{flipping a square}, if there exists some $1 \leq i \leq n-1$ such that $$\gamma'= e_1 \cdots e_{i-1} \cdot a \cdot b \cdot e_{i+2} \cdots e_n,$$ where $e_i,e_{i+1},b,a$ define a 4-cycle in $X$;
	\item \emph{shortening a triangle}, if there exists some $1 \leq i \leq n-1$ such that $$\gamma'= e_1 \cdots e_{i-1} \cdot a \cdot e_{i+2} \cdots e_n,$$ where $e_i,e_{i+1},a$ define a 3-cycle in $X$;
	\item \emph{removing a backtrack}, if there exists some $1 \leq i \leq n-1$ such that $$\gamma'= e_1 \cdots e_{i-1} \cdot e_{i+2} \cdots e_n,$$ where $e_{i+1}$ is the inverse of $e_i$.
\end{itemize}
Then, given two path with the same endpoints, one can be obtained from the other by flipping squares, shortening triangles, removing backtracks, and inverses of these operations. (See for instance \cite[Lemma~2.2]{SpecialRevisited}.)

\medskip \noindent
Now, given a path from $a_1$ to $a_2$, we colour its edges successively such that the colour changes exactly when we pass through $x$ from one clique to a second clique that lies in a distinct connected component in the clique-link of $x$. The number of colours turns out to be an invariant: applying any of our elementary operations to a path does not modify the number of colours. Because $\epsilon_1 \cup \epsilon_2$ has two colours, it turns out that there cannot exist a path from $a_1$ to $a_2$ that does not pass through $x$ because such a path would have only one colour. Thus, we have proved that $x$ is a cut-vertex, concluding the proof of Claim~\ref{claim:CutVertex}. 

\medskip \noindent
We are finally ready to prove our lemma. If $X$ has a cut-vertex $x$, then two hyperplanes lying in distinct components of $X \backslash \{x\}$ cannot be transverse, so the crossing graph is clearly not connected. Conversely, assume that $X$ has no cut-vertex. Let $A,B$ be two hyperplanes. We fix an arbitrary path from $N(A)$ to $N(B)$, and we denote by $J_1, \ldots, J_n$ the hyperplanes successively crossed. Observe that $A,J_1, \ldots, J_n,B$ are successively in contact. But the fact that the clique-link of every vertex is connected, which follows from Claim~\ref{claim:CutVertex}, implies that two hyperplanes in contact are connected by a path in the crossing graph. Therefore, there exists path in $\Delta X$ connecting $A$ and $B$. This proves that $\Delta X$ is connected.
\end{proof}

\noindent
Even when a quasi-median graph has no cut-vertex, the crossing and contact graphs may be quite different. In fact, according to \cite[Proposition~2.19]{MR3217625}, given any graph, it is possible to construct a median graph whose crossing graph is this particular graph. On the other hand, the contact graph is always a quasi-tree \cite{MR3217625}. Nevertheless:

\begin{prop}
Let $X$ be a quasi-median graph. Assume that $X$ has no cut-vertex and that every vertex in $X$ belongs to $\leq N$ cliques for some fixed $N$. Then the canonical map $\Delta X \to \Gamma X$ is a quasi-isometry. 
\end{prop}

\noindent
The same statement for median graphs can be found in \cite{MR4057355}. Our proposition follows from Lemmas~\ref{lem:DistGamma} and~\ref{lem:DistDelta} below, but it will not be used in this paper.

\subsection{Hyperbolicity}

\noindent
As already mentioned, the contact graph of a median graph is always a quasi-tree \cite{MR3217625}. The same statement, using the same arguments, can be proved for quasi-median graphs. However, this fact will not be necessary in the rest of the paper. Instead, we show that crossing and contact graphs are hyperbolic and compute explicit hyperbolicity constants. We start with the contact graph:

\begin{lemma}\label{lem:Contact}
Let $X$ be a quasi-median graph and $A,B$ two hyperplanes. If $d_{\Gamma X}(A,B) \geq 2$, then every vertex of every geodesic in $\Gamma X$ between $A$ and $B$ lies at distance $\leq 2$ from a hyperplane separating $A$ and $B$. 
\end{lemma}

\begin{proof}
Let $H_1, \ldots, H_m$ be a maximal collection of pairwise non-transverse hyperplanes separating $A$ and $B$; for convenience, we set $H_0:=A$ and $H_{m+1}:=B$. Up to re-indexing our hyperplanes, we assume that $H_i$ separates $H_{i-1}$ and $H_{i+1}$ for every $1 \leq i \leq m$. Notice that $m \geq 1$ because $d_{\Gamma X} (A,B) \geq 2$. Fix a geodesic $J_1, \ldots, J_n$ from $A$ to $B$ in $\Gamma X$ and an index $1 \leq i \leq n$. We claim that $J_i$ lies at distance $\leq 2$ from $H_j$ for some $1 \leq j \leq m$.

\medskip \noindent
If $J_i$ is in contact with $H_j$ for some $0 \leq j \leq m+1$, then there exists some $1 \leq k \leq m$ such that $d_{\Gamma X}(J_i, H_k) \leq 2$. From now on, we assume that $J_i$ is not in contact with any $H_j$. As a consequence, $J_i$ lies between $H_j$ and $H_{j+1}$ for some $0 \leq j \leq m$. Let $r$ be the first index $\ell < i$ such that $J_\ell$ is in contact with $H_j$ and let $s$ be the last index $\ell>i$ such that $J_\ell$ is in contact with $H_{j+1}$. We have
$$d_{\Gamma X} (J_r,J_s) \leq 2+ d_{\Gamma X}(H_j,H_{j+1}) = 3.$$
Consequently, $J_k$ lies at distance $\leq 3/2 + 1$ from $H_j$ or $H_{j+1}$. 
\end{proof}

\begin{cor}
Let $X$ be a quasi-median graph. Its contact graph $\Gamma X$ is $3$-hyperbolic.
\end{cor}

\begin{proof}
Let $A,B,C$ be three hyperplanes. Fix three geodesics $[A,B]$, $[B,C]$, $[A,C]$ in $\Gamma X$ and a vertex $H \in [A,B]$. If $d_{\Gamma X}(A,B) \leq 1$, then $H$ lies in $[B,C] \cup [A,C]$. Otherwise, we know from Lemma~\ref{lem:Contact} that $d_{\Gamma X}(H,J) \leq 2$ for some hyperplane $J$ separating $A$ and $B$. If $J$ is transverse to $C$, then $A$ lies at distance $\leq 3$ from $[B,C] \cup [A,C]$. Otherwise, $J$ separates $B$ and $C$ or $A$ and $C$ (not exclusively). Necessarily, there exists some hyperplane in $[B,C]$ or $[A,C]$ equal or transverse to $J$, so $H$ lies at distance $\leq 3$ from $[B,C] \cup [A,C]$. Thus, we have proved that $[A,B]$ lies in the $3$-neighbourhood of $[B,C] \cup [A,C]$. 
\end{proof}

\noindent
Next, we turn to the crossing graph. Recall that, in full generality, the crossing graph may be any graph, and in particular it may not be hyperbolic. Nevertheless, it becomes hyperbolic (and in fact a quasi-tree) under the good assumptions:

\begin{lemma}\label{lem:Crossing}
Let $X$ be a quasi-median graph and $A,B$ two hyperplanes. Assume that every $X$ has no cut-vertex and that every vertex belongs to at most $N$ cliques. If $d_{\Delta X}(A,B) \geq N$, then every vertex of every geodesic in $\Delta X$ between $A$ and $B$ lies at distance $\leq 2+N/2$ from a hyperplane separating $A$ and $B$. 
\end{lemma}

\begin{proof}
Let $H_1, \ldots, H_m$ be a maximal collection of pairwise non-transverse hyperplanes separating $A$ and $B$; for convenience, we set $H_0:=A$ and $H_{m+1}:=B$. Up to re-indexing our hyperplanes, we assume that $H_i$ separates $H_{i-1}$ and $H_{i+1}$ for every $1 \leq i \leq m$. Notice that $m \geq 1$ because $d_{\Gamma X} (A,B) \geq N$. (Indeed, two hyperplanes in contact lies at distance $\leq N$ in $\Delta X$.) Fix a geodesic $J_1, \ldots, J_n$ from $A$ to $B$ in $\Delta X$ and an index $1 \leq i \leq n$. We claim that $J_i$ lies at distance $\leq 2$ from $H_j$ for some $1 \leq j \leq m$.

\medskip \noindent
If $J_i$ is equal or transverse to $H_j$ for some $0 \leq j \leq m+1$, then there exists some $1 \leq k \leq m$ such that $d_{\Gamma X}(J_i, H_k) \leq 2$. From now on, we assume that $J_i$ is distinct and not transverse to any $H_j$. As a consequence, $J_i$ lies between $H_j$ and $H_{j+1}$ for some $0 \leq j \leq m$. Let $r$ be the first index $\ell < i$ such that $J_\ell$ is transverse to $H_j$ and let $s$ be the last index $\ell>i$ such that $J_\ell$ is transverse to $H_{j+1}$. We have
$$d_{\Delta X} (J_r,J_s) \leq 2+ d_{\Delta X}(H_j,H_{j+1}) \leq 2+N.$$
Consequently, $J_k$ lies at distance $\leq (2+N)/2 + 1$ from $H_j$ or $H_{j+1}$. 
\end{proof}

\begin{cor}
Let $X$ be a quasi-median graph. Assume that every $X$ has no cut-vertex and that every vertex belongs to at most $N$ cliques. The crossing graph $\Delta X$ is $(3+N/2)$-hyperbolic.
\end{cor}

\begin{proof}
Let $A,B,C$ be three hyperplanes. Fix three geodesics $[A,B]$, $[B,C]$, $[A,C]$ in $\Delta X$ and a vertex $H \in [A,B]$. If $d_{\Delta X}(A,B) <N$, then $H$ lies in the $(N-1)/2$-neighbourhood of $[B,C] \cup [A,C]$. Otherwise, we know from Lemma~\ref{lem:Crossing} that $d_{\Gamma X}(H,J) \leq 2+N/2$ for some hyperplane $J$ separating $A$ and $B$. If $J$ is transverse to $C$, then $A$ lies at distance $\leq 3+N/2$ from $[B,C] \cup [A,C]$. Otherwise, $J$ separates $B$ and $C$ or $A$ and $C$ (not exclusively). Necessarily, there exists some hyperplane in $[B,C]$ or $[A,C]$ equal or transverse to $J$, so $H$ lies at distance $\leq 3+N/2$ from $[B,C] \cup [A,C]$. Thus, we have proved that $[A,B]$ lies in the $(3+N/2)$-neighbourhood of $[B,C] \cup [A,C]$. 
\end{proof}

\subsection{Strongly contracting isometries}

\noindent
Finally, we turn to the question of when a given isometry of a quasi-median graph induces an isometry with unbounded orbits in the corresponding crossing or contact graph. We refer to these isometries as \emph{strongly contracting}. 

\begin{prop}\label{prop:Loxo}
Let $X$ be a quasi-median graph and $g \in \mathrm{Isom}(X)$ an isometry admitting an axis $\gamma$. Let $\Omega X$ denote the crossing or contact graph of $X$. In the former case, assume that $X$ has no cut-vertex and that every vertex belongs to $\leq N$ cliques. The following assertions are equivalent:
\begin{itemize}
	\item $g$ has unbounded orbits in $\Omega X$;
	\item $g$ is a loxodromic isometry of $\Omega X$;
	\item there exist a power $r \geq 1$ and a hyperplane $J$ crossing $\gamma$ such that $\{g^{rk}J \mid k \in \mathbb{Z}\}$ is a collection of pairwise strongly separated hyperplanes.
\end{itemize}
When these conditions are satisfied, we say that $g$ is \emph{strongly separated}. 
\end{prop}

\noindent
Recall that two hyperplanes are \emph{strongly separated} if no hyperplane can be transverse to both of them. Given two hyperplanes $A$ and $B$, we denote by $\mathrm{ss}(A,B)$ the maximal number of pairwise strongly separated hyperplanes separating $A$ and $B$.

\medskip \noindent
Our proposition will be a straightforward consequence of the following estimates of distances in crossing and contact graphs.

\begin{lemma}\label{lem:DistGamma}
Let $X$ be a quasi-median graph. The inequalities
$$\mathrm{ss}(A,B) \leq d_{\Gamma X}(A,B) \leq 3( 1+ \mathrm{ss}(A,B))$$
hold for all hypeprlanes $A,B$. 
\end{lemma}

\begin{lemma}\label{lem:DistDelta}
Let $X$ be a quasi-median graph with no cut-vertex and all of whose vertices belong to $\leq N$ cliques. The inequalities
$$\mathrm{ss}(A,B) \leq d_{\Delta X}(A,B) \leq (2+N) ( 1+ \mathrm{ss}(A,B))$$
hold for all hyperplanes $A,B$. 
\end{lemma}

\begin{proof}[Proof of Lemmas~\ref{lem:DistGamma} and~\ref{lem:DistDelta}.]
Let $J_1, \ldots, J_n$ be a maximal collection of pairwise strongly separated hyperplanes separating $A$ and $B$. For convenience, set $J_0:=A$ and $J_{n+1}:=B$. Given an $0 \leq i \leq n$, observe that the last hyperplane in a maximal collection of pairwise non-transverse hyperplanes separating $J_i$ from $J_{i+1}$ cannot be strongly separated together with $J_i$ (by maximality of our collection $J_1, \ldots, J_n$) and it must be in contact with $J_{i+1}$. Hence $d_{\Gamma X}(J_i,J_{i+1}) \leq 2+1=3$ and $d_{\Delta X}(J_i,J_{i+1}) \leq 2 + N$. Therefore,
$$d_{\Gamma X} (A,B) \leq \sum\limits_{i=0}^n d_{\Gamma X}(J_i,J_{i+1}) \leq 3(n+1)= 3 (1+ \mathrm{ss}(A,B) )$$
and
$$d_{\Delta X} (A,B) \leq \sum\limits_{i=0}^n d_{\Delta X} (J_i,J_{i+1}) \leq (2+N) (n+1)= (2+N)(1+ \mathrm{ss}(A,B)).$$
Next, let $H_1, \ldots, H_m$ be a geodesic in $\Omega X$ ($= \Delta X$ or $\Gamma X$) between $A$ and $B$. For every $1 \leq i \leq n$, there must exist some $1 \leq j \leq m$ such that $H_j$ is equal or transverse to $J_i$. Because the $J_i$ are pairwise strongly separated, a $H_j$ cannot appear twice, hence $m \geq n$. We conclude that $d_{\Omega X}(A,B) \geq \mathrm{ss}(A,B)$. 
\end{proof}

\begin{proof}[Proof of Proposition \ref{prop:Loxo}.]
Let $J$ be a hyperplane crossing $\gamma$. If $g$ has unbounded orbits, then there exists some $r \geq 1$ such that $d_{\Omega X}(J,g^rJ) \geq 3$. Necessarily, $J$ and $g^r J$ are strongly separated, which implies that $\{ g^{rk} J \mid k \in \mathbb{Z}\}$ is a collection of pairwise strongly separated hyperplanes. We deduce easily from Lemmas~\ref{lem:DistGamma} and~\ref{lem:DistDelta} that $k \mapsto g^{rk} J$ defines a quasi-isometric embedding $\mathbb{Z} \to \Omega X$, so $g$ is loxodromic in $\Omega X$. Of course, this implies that $g$ has unbounded orbits in $\Omega X$, concluding the proof of our proposition. 
\end{proof}

\section{Axes and rational translation lengths}

\subsection{General case}

\noindent
We are now ready to prove the first main theorem of this article, namely that isometries of quasi-median graphs induce isometries of the crossing / contact graph with rational (asymptotic) translation lengths. Our next statement proves Theorem~\ref{thm:IntroAxis} and Corollary~\ref{cor:IntroRational} from the introduction. 

\begin{thm}\label{thm:Rational}
Let $X$ be a quasi-median graph, and let $\Omega X$ be either the crossing graph or the contact graph of $X$. In the latter case, we assume that $X$ has no cut-vertices; and, in any case, we assume that a vertex of $X$ always belongs to $\leq N$ cliques. For every isometry $g \in \mathrm{Isom}(X)$ with unbounded orbits in $\Omega X$, there exists a positive $k \leq N^{\mathrm{QC}(g)}$ such that $g^k$ admits an axis in $\Omega X$. Moreover, there exists such an axis containing a hyperplane in contact with $B(o,2\mathrm{QC}(\gamma)+1) \cap  \mathrm{CH}(\gamma)$ where $o$ is an arbitrary vertex of an arbitrary axis $\gamma$ of $g$ in $X$. 
\end{thm}

\noindent
Our proof is based on the next criterion, which originates from \cite{MR1390660}. We refer to \cite{Rational} for a proof of the version we use here.

\begin{prop}\label{prop:FindingAxis}
Let $Y$ be a graph and $g \in \mathrm{Isom}(Y)$ an isometry. Assume that there exists an isometrically embedded subgraph $A \subset Y$ on which $\langle g \rangle$ acts geometrically. Then $g$ cyclically permutes $\leq \mathrm{width}(A)$ pairwise disjoint geodesics in $A$. Consequently, there exists some $k \leq \mathrm{width}(A)$ such that $g^k$ has an axis in $Y$ and the translation length of $g$ is a rational number with $k$ as a denominator.
\end{prop}

\noindent
Recall that the \emph{width} of a quasi-line is the minimal size of a set of vertices separating the two ends of the quasi-line.

\begin{proof}[Proof of Theorem~\ref{thm:Rational}.]
According to Proposition~\ref{prop:Axis}, $g$ admits an axis $\gamma$ in $X$. 

\begin{claim}\label{claim:One}
The convex hull $\mathrm{CH}(\gamma)$ of $\gamma$ is locally finite and $\langle g \rangle$ acts cocompactly on it.
\end{claim}

\noindent
Assume for contradiction that $\mathrm{CH}(\gamma)$ is not locally finite, i.e.\ there exists a vertex $x \in \mathrm{CH}(\gamma)$ with infinitely many neighbours in $\mathrm{CH}(\gamma)$. As a consequence, since every vertex in $X$ belongs to only finitely many cliques, we can find two distinct neighbours $y,z \in \mathrm{CH}(\gamma)$ of $x$ such that $x,y,z$ all belong to a common clique $C$. Let $J$ denote the hyperplane containing $C$. Because $x,y,z$ belong to three pairwise distinct sectors delimited by $J$, it follows from Corollary~\ref{cor:ConvexHull} that $\gamma$ intersects these sectors, which is impossible since we know from Theorem~\ref{thm:BigQM} that $\gamma$ cannot intersect $J$ twice. Thus, we have proved that $\mathrm{CH}(\gamma)$ is locally finite.

\medskip \noindent
In order to prove that $\langle g \rangle$ acts cocompactly on $\mathrm{CH}(\gamma)$, it suffices to show that $\mathrm{CH}(\gamma)$ lies in a neighbourhood of $\gamma$. So we fix a vertex $x \in \mathrm{CH}(\gamma)$. According to Proposition~\ref{prop:Loxo}, there exist a power $r \geq 1$ and a hyperplane $J$ crossing $\gamma$ such that $\{g^{rk}J \mid k \in \mathbb{Z} \}$ is a collection of pairwise strongly separated hyperplanes. Up to translating $J$ with a power of $g$, we assume that $x$ lies between $J$ and $g^r J$. Let $y \in \gamma$ be an arbitrary vertex lying between $J$ and $g^rJ$. Observe that, as a consequence of Corollary~\ref{cor:ConvexHull}, every hyperplane separating $x$ and $y$ must cross $\gamma$. On the other hand, a hyperplane crossing $\gamma$ before $g^{-r}J$ (resp. after $g^{2r}J$) cannot separate $x$ and $y$ since otherwise it would be transverse to $g^{-1}J$ and $J$ (resp. to $g^rJ$ and $g^{2r}J$). Consequently, all the hyperplanes separating $x$ and $y$ must cross $\gamma$ between $g^{-r}J$ and $g^{2r}J$. It follows from Theorem~\ref{thm:BigQM} that $d(x,y) \leq 3r\|g\|$ where $\|g\|$ denotes the translation length of $g$ along $\gamma$. We conclude that $\mathrm{CH}(\gamma)$ lies in the $3r\|g\|$-neighbourhood of $\gamma$, as desired. The proof of Claim~\ref{claim:One} is complete.

\medskip \noindent
Let $A(g)$ denote all the hyperplanes of $X$ in contact with $\mathrm{CH}(\gamma)$. 

\begin{claim}\label{claim:Two}
The action $\langle g \rangle \curvearrowright A(g)$ is cocompact.
\end{claim}

\noindent
It follows from the assumption that every vertex in $X$ belongs to only finitely many cliques and from Claim~\ref{claim:One} that $\langle g \rangle$ acts on $A(g)$ with only finitely many orbits of vertices. In order to show that there are only finitely many orbits of edges too, it is sufficient to observe that $A(g)$, as a subgraph of the crossing graph, is locally finite. So fix a hyperplane $H$ in $A(g)$, i.e. tangent to $\mathrm{CH}(\gamma)$. According to Proposition~\ref{prop:Loxo}, there exist a power $r \geq 1$ and a hyperplane $J$ crossing $\gamma$ such that $\{g^{rk}J \mid k\in \mathbb{Z}\}$ is a collection of pairwise strongly separated hyperplanes. Up to translating $J$ by a power of $g$, we assume that $H$ lies between $J$ and $g^{2r}J$. Clearly, a hyperplane in contact with $H$ must lie between $g^{-r}J$ and $g^{3r}J$. Therefore, the neighbours of $H$ in $A(g)$ are hyperplanes in contact with the (finite) piece of $\mathrm{CH}(\gamma)$ lying between $g^{-r}J$ and $g^{3r}J$, which yields the desired conclusion. The proof of Claim~\ref{claim:Two} is complete.

\begin{claim}\label{claim:Three}
$A(g)$ is isometrically embedded in $\Omega X$.
\end{claim}

\noindent
Let $M,N \in A(g)$ be two hyperplanes. Fix an arbitrary geodesic $J_1, \ldots, J_k$ from $M$ to $N$ in $\Omega X$ and define its complexity as $d(N(J_1),\mathrm{CH}(\gamma))+ \cdots + d(N(J_k), \mathrm{CH}(\gamma))$. If the complexity is zero, then our geodesic lies in $A(g)$ and there is nothing to prove. From now on, we assume that the complexity is not zero, i.e. there exists some $2 \leq i \leq k-1$ such that $N(J_i)$ and $\mathrm{CH}(\gamma)$ are disjoint. According to Proposition~\ref{prop:separation}, there exists a hyperplane $J$ weakly separating $N(J_i)$ and $\mathrm{CH}(\gamma)$. Necessarily, there exist $1 \leq a \leq i-1$ and $i+1 \leq b \leq k$ such that $J$ is transverse to $J_a$ and $J_b$. Observe that $b-a \leq 2$ since otherwise it would be possible to shorten our geodesic by replacing $J_a, J_{a+1}, \ldots, J_{b-1},J_b$ with $J_a,J,J_b$. In other words, $J$ is transverse to $J_{i-1}$ and $J_{i+1}$. By replacing $J_i$ with $J$, we obtain a new path from $M$ to $N$ that is a geodesic, since the length remains the same, and whose complexity is smaller, since $d(\mathrm{CH}(\gamma),N(J)) < d(\mathrm{CH}(\gamma),N(J_i))$. By iterating, we eventually get a geodesic lying in $A(g)$, concluding the proof of Claim~\ref{claim:Three}.

\begin{claim}\label{claim:Four}
Fix a finite set of vertices $S \subset \mathrm{CH}(\gamma)$ separating the ends of $\mathrm{CH}(\gamma)$ and let $CS$ denote the set of the hyperplanes in contact with $S$. Then $CS$ separates the ends of $A(g)$. 
\end{claim}

\noindent
Let $A,B$ be two hyperplanes such that the intersections between their carriers and $\mathrm{CH}(\gamma)$ are non-empty and separated by $S$. Let $J_1, \ldots, J_n$ be a path in $\Omega X$ from $A$ to $B$ all of whose hyperplanes are in contact with $\mathrm{CH}(\gamma)$. If one of the $J_i$ is in contact with $S$, we are done. Otherwise, each $J_i$ can be labelled by the component of $\mathrm{CH}(\gamma) \backslash S$ containing $N(J_i) \cap \mathrm{CH}(\gamma)$. Because $A=J_1$ and $B=J_n$ have different labels, there must exist some $1 \leq i \leq n-1$ such that $J_i$ and $J_{i-1}$ have different labels. Fix three vertices $a \in N(J_i) \cap \mathrm{CH}(\gamma)$, $b \in N(J_{i+1}) \cap \mathrm{CH}(\gamma)$, and $c \in N(J_i) \cap N(J_{i+1})$. Let $(x,y,z)$ be the median triangle of $(a,b,c)$. Because $N(J_i)$ is gated, $a,x,z,c,b$ all belong to $N(J_i)$. Similarly, $b,y,z,c,a$ all belong to $N(J_{i+1})$. Thus, there exists a geodesic from $a$ to $b$ passing through $x$ and $y$ lying in $(N(J_i) \cup N(J_{i+1})) \cap \mathrm{CH}(\gamma)$. Because $a$ and $b$ are separated by $S$, it follows that either $J_i$ or $J_{i+1}$ is in contact with $S$. This concludes the proof of Claim~\ref{claim:Four}. 

\medskip \noindent
In order to conclude the proof of our theorem thanks to Proposition~\ref{prop:FindingAxis} and our previous claims, it suffices to notice that, given an arbitrary vertex $o \in \gamma$, $B(o,2\mathrm{QC}(\gamma)+1) \cap \mathrm{CH}(\gamma)$ separates the ends of $\gamma$, which follows from our next general observation:

\begin{claim}
Let $Y$ be a graph, $Z \subset Y$ an isometrically embedded subgraph, and $\xi \subset Z$ a bi-infinite geodesic. Assume that the Hausdorff distance $D$ between $\xi$ and $Z$ is finite. Given a vertex $o \in \xi$, the ball $B(o,2D+1) \cap Z$ separates the ends of $Z$. 
\end{claim}

\noindent
Let $p,q \in Z$ be two vertices separated by a large ball centred at $o$. Fix a geodesic $\zeta$ from $p$ to $q$. Every vertex of $\zeta$ lies at distance $\leq D$ from a vertex of $\xi$. If such a vertex lies in the same component $\xi_\ell$ (resp. $\xi_r$) of $\xi \backslash B(o,D+1)$ as $p$ (resp. $q$), we say that it is a left (resp. right) vertex. Let $a \in \zeta$ denote the last left vertex along $\zeta$ and let $b$ denote the first right vertex following $a$. If $x \in \xi_\ell$ (resp. $y \in \xi_r$) is vertex of $\xi$ at distance $\leq D$ from $a$ (resp. $b$), then
$$2(D+1) \leq d(x,y) \leq d(x,a)+d(a,b) + d(b,y) \leq d(a,b)+2D,$$
hence $d(a,b) \geq 2$. Consequently, there exists some vertex $c \in \zeta$ between $a$ and $b$. By construction, $c$ is neither left nor right, so there exists some $z \in \xi \cap B(o,D+1)$ such that $d(c,z) \leq D$. Then
$$d(c,o) \leq d(c,z)+d(z,o) \leq D+D+1=2D+1,$$
proving that $\zeta$ intersects $B(o,2D+1)$, as desired. 
\end{proof}

\subsection{Hyperbolic case}

\noindent
Given an isometry $g$ of some quasi-median graph $X$, the upper bound given by Theorem~\ref{thm:Rational} depends only on the local structure of $X$ and the quasiconvexity constant $\mathrm{QC}(g)$ of $g$. In full generality, the quantity $\mathrm{QC}(g)$ can take arbitrarily large finite values. But there is a worth mentioning situation where the quasiconvexity constants are uniformly bounded: in hyperbolic graphs. Indeed, if $X$ is $\delta$-hyperbolic, then the Hausdorff distance between any two bi-infinite geodesics having the same endpoints at infinity is at most $8\delta$ (see for instance \cite[Proposition~2.2.2]{MR1075994}), hence $\mathrm{QC}(g) \leq 8 \delta$ for every loxodromic isometry $g \in \mathrm{Isom}(X)$. Thus, Theorem~\ref{thm:Rational} immediately implies:

\begin{cor}\label{cor:HyperbolicCase}
Let $X$ be a $\delta$-hyperbolic quasi-median graph, and let $\Omega X$ be either the crossing graph or the contact graph of $X$. In the latter case, we assume that $X$ has no cut-vertices; and, in any case, we assume that a vertex of $X$ always belongs to $\leq N$ cliques. For every isometry $g \in \mathrm{Isom}(X)$ with unbounded orbits in $\Omega X$, there exists a positive $k \leq N^{8\delta}$ such that $g^k$ admits an axis in $\Omega X$. As a consequence, the translation length of $g$ in $\Omega X$ is a rational with a denominator $\leq N^{8\delta}$. 
\end{cor}

\subsection{Some examples}\label{section:Examples}

\noindent
In this section, we describe examples where isometries of (quasi-)median graphs can have arbitrarily small translation lengths in the corresponding crossing graphs, thus contrasting with Corollary~\ref{cor:HyperbolicCase}.

\medskip \noindent
\begin{minipage}{0.38\linewidth}
\includegraphics[width=0.9\linewidth]{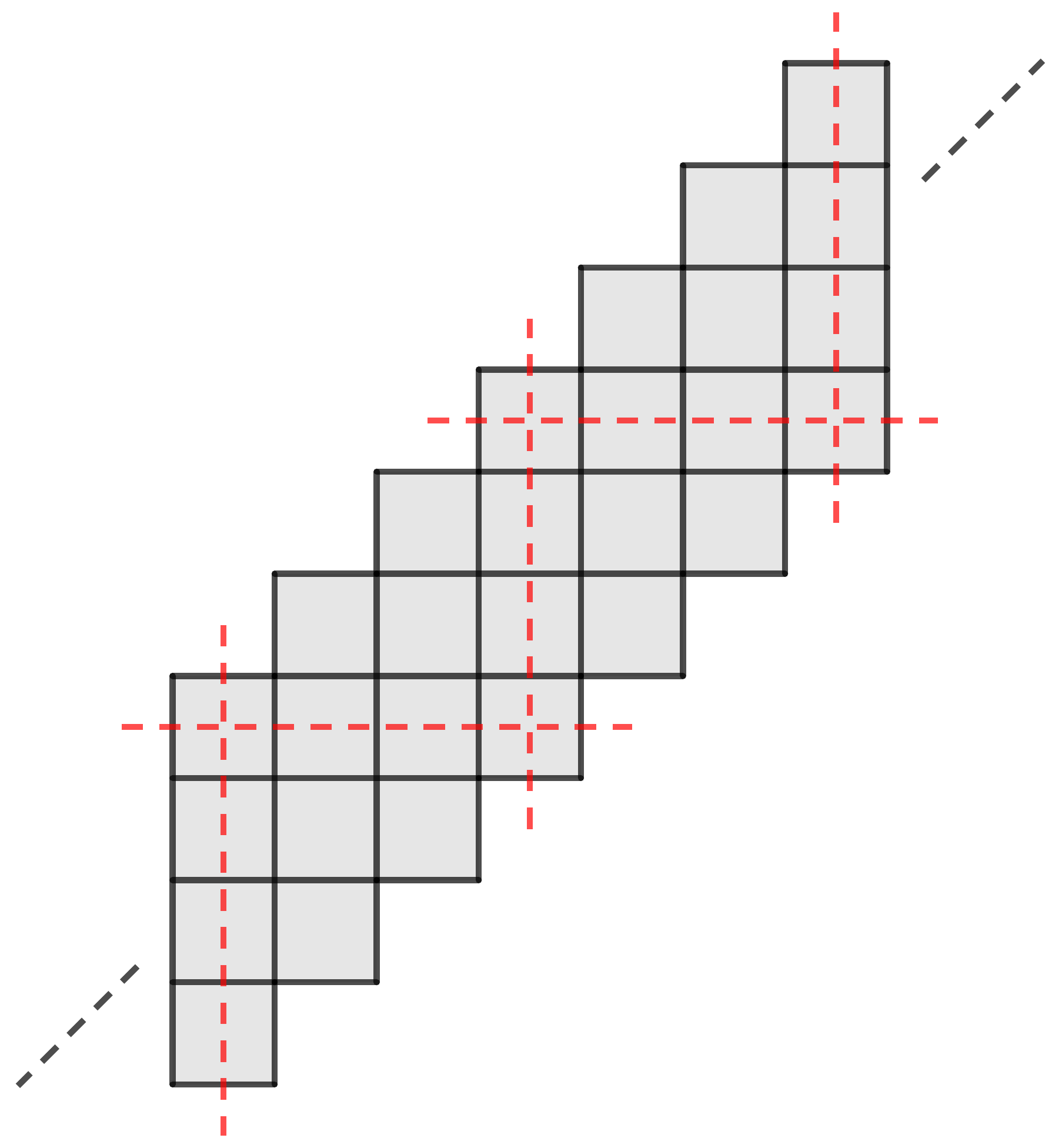}
\end{minipage}
\begin{minipage}{0.6\linewidth}
Fix an $n \geq 1$ and let $X_n$ denote the median graph given by the subgraph in $\mathbb{Z}^2$ delimited by a regular staircase and its translate under $(0,n)$. There is an obvious isometry $g \in \mathrm{Isom}(X_n)$ of translation length two. The figure on the left represents $X_3$. If $J$ denotes a vertical hyperplane in $X_n$, then $J$ belongs to an axis of $g^n$ in the crossing graph $\Delta X_n$ on which $g^n$ acts as a translation of length two. Consequently, $g$ has translation length $2/n$ in $\Delta X_n$. 
\end{minipage}

\medskip \noindent
From this example, we can artificially create a median graph admitting isometries with arbitrarily small but positive translation lengths in the crossing graph. For instance, for every $n \geq 1$ let $Z_n$ be an infinite cyclic group. The free product $G:= Z_1 \ast Z_2 \ast \cdots$ acts on its Bass-Serre tree $T$. We can blow up each vertex of $T$ stabilised by a conjugate of $Z_n$ and replace it with a copy of $X_n$ on which $Z_n$ acts. Then $G$ acts on the median graph $X$ thus obtained, and, for every $n \geq 1$, a generator of $Z_n$ acts on the crossing graph of $X$ with a translation length equal to $2/n$. 

\medskip \noindent
More interestingly, the example constructed in \cite{NoFactorSystem} produces a group acting properly and cocompactly on a median graph and containing elements with arbitrarily small translation length in the crossing graph. This observation is a straightforward consequence of the arguments in \cite{NoFactorSystem}. We describe the construction for the reader's convenience. 

\medskip \noindent
Let $H$ be a group acting properly and cocompactly on a median graph $Y$. We assume that there exist elements $g \in H$ and $h_1,h_2, \ldots \in H$ such that:
\begin{itemize}
	\item $g$ admits a convex axis $\gamma$;
	\item for every $n \geq 1$, the projection of $h_n \gamma$ on $\gamma$ has finite length $\ell_n \geq n$. 
\end{itemize}
Such examples can be found, for instance, in uniform lattices of products of trees. Given an $n \geq 1$, up to conjugating $h_n$ with a power of $g$, we can assume that $k_n \cdot \mathrm{proj}_{\gamma}( h_n \gamma)$ intersects $\mathrm{proj}_{h_n \gamma}(\gamma)$ along a proper subsegment of length $s_n> \ell_n-\|g\|$. 

\medskip \noindent
Consider the HNN extension $G := \langle H,t \mid tgt^{-1}=t \rangle$ acting on the median graph $X$ obtained from the tree of spaces modelled on the Bass-Serre tree whose vertex-spaces are copies of $Y$ and whose edge-spaces are copies of $\gamma \times [0,1]$. For every $n \geq 1$, set $k_n:=th_n$. The configuration to keep in mind is illustrated by Figure~\ref{Sam}.

\begin{figure}[h!]
\begin{center}
\includegraphics[width=0.7\linewidth]{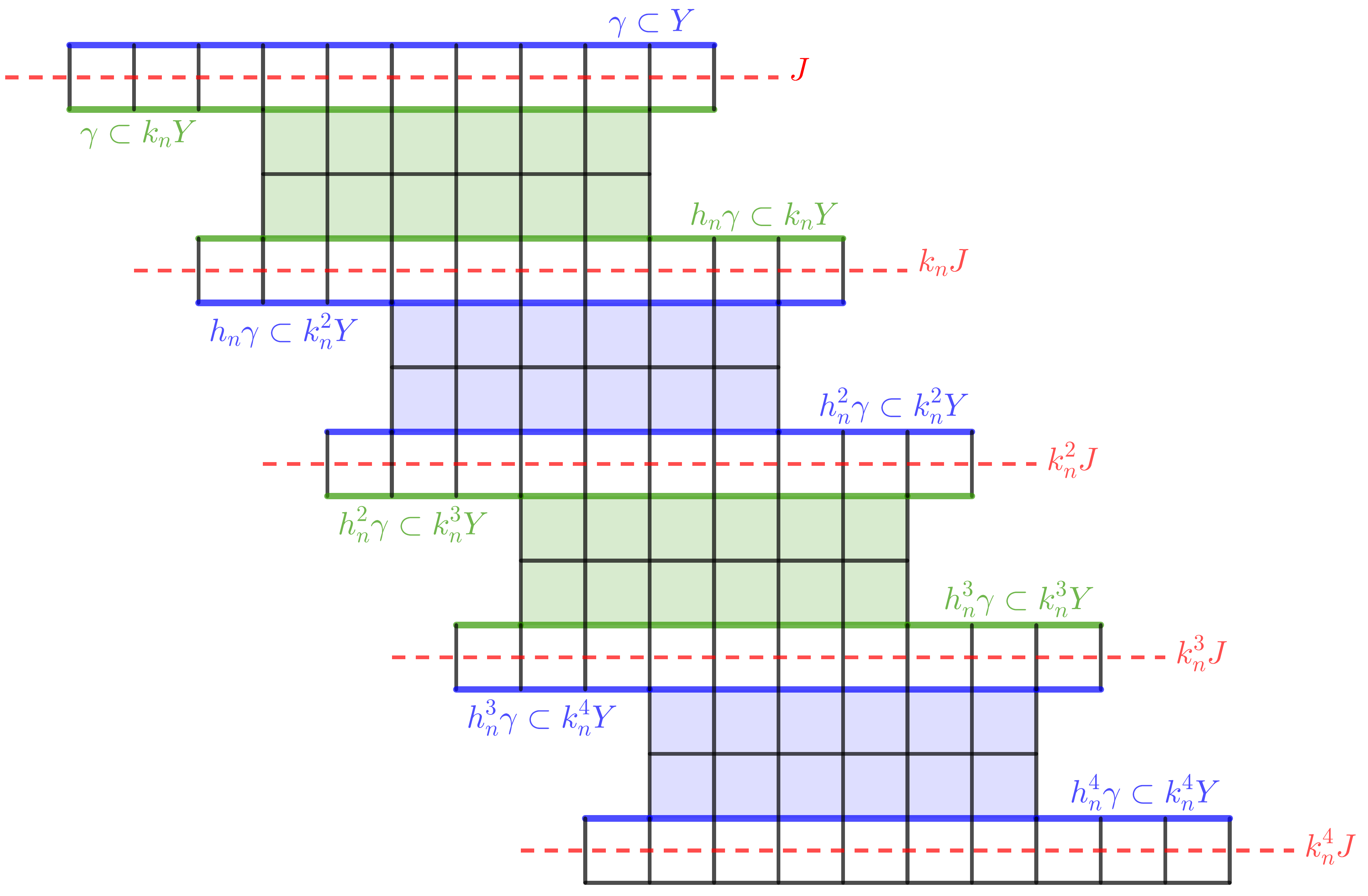}
\caption{The isometry $k_n$ has a small translation length in the crossing graph.}
\label{Sam}
\end{center}
\end{figure}

\noindent
Set $r:= \left\lfloor \frac{\ell_n}{ \ell_n-s_n} \right\rfloor$. One easily checks that $d_{\Delta X}(J, k_n^rJ)=2$ and that $k_n^{-r}J, J,k_n^rJ,k_n^{2r}J, \ldots$ all lie on a bi-infinite geodesic in $\Delta X$. We have
$$r = \left\lfloor \frac{\ell_n}{ \ell_n-s_n} \right\rfloor > \left\lfloor \frac{n}{\|g\|} \right\rfloor, \text{ hence } \|k_n \|_{\Delta X} = \frac{2}{r} \leq \frac{2\|g\|}{n}.$$
Thus, the element $k_n$ of $G$ has smaller and smaller translation length in $\Delta X$ as $n$ goes to infinity.

\section{Constructible quasi-median graphs}\label{section:Constructible}

\noindent
In this section, our goal is to show that translation lengths in crossing and contact graphs are algorithmically computable. Of course, quasi-median graphs and isometries have to be given algorithmically in some sense. In order to clarify such an assumption, we introduce \emph{constructible} quasi-median graphs and \emph{computable} isometries.

\begin{definition}
Let $X$ be a quasi-median graph. 
\begin{itemize}
	\item $X$ is \emph{geodesically constructible} if there exists an algorithm that determines, given two vertices $x,y \in X$, the interval $I(x,y)$.
	\item $X$ is \emph{locally constructible} if there exists an algorithm that yields, given a vertex $x \in X$, a set of neighbours $\mathcal{N}(x)$ such that every clique containing $x$ intersects $\mathcal{N}(x)$ along exactly one vertex. 
\end{itemize}
A quasi-median graph is \emph{constructible} if it is geodesically and locally constructible. 
\end{definition}

\noindent
It could be tempted to define a graph $X$ as being \emph{metrically constructible} if there exists an algorithm that constructs, given a vertex $x$ and an integer $R \geq 0$, the ball $B(x,R)$. However, despite the fact that metrically constructible are constructible, this notion only deals with locally finite graphs: since an algorithm has to stop after a finite amount of time, its output must contain only a finite amount of data. Because we are also interested in quasi-median graphs that are not locally finite (e.g.\ quasi-median graphs associated to right-angled Artin groups), metrically constructible graphs are not relevant for us. Nevertheless, our quasi-median graphs do satisfy some weak property of local finiteness: every vertex belongs to only finitely many cliques. This is the property exploited by our definition of locally constructible quasi-median graphs. It is worth noticing, however, that being constructible and being metrically constructible coincide for median graphs since their cliques are edges. 

\begin{definition}
Let $X$ be a graph. An isometry $g \in \mathrm{Isom}(X)$ is \emph{computable} if there exists an algorithm that determines, given a vertex $x \in X$, the translate $g \cdot x$. 
\end{definition}

\noindent
As a justification of the well-founded of our definitions, let us mention that, given a compact \emph{locally quasi-median prism complex} $Q$ as defined in \cite{QM} (e.g.\ a compact nonpositively curved cube complex), the one-skeleton $X$ of the universal cover of $Q$ is a constructible quasi-median graph and every $g \in \pi_1(Q)$ induces on $X$ a computable isometry. This observation, which will not be used in our article, essentially follows from \cite[Lemma~2.2]{QM} (already used during the proof of Lemma~\ref{lem:CrossingConnected}), which shows that any two homotopic paths in the one-skeleton of $Q$ can be transformed into each other by a sequence of elementary transformations.

\subsection{Basic algorithmic problems}

\noindent
In this section, we record a few elementary problems in constructible quasi-median graphs that can be solved algorithmically. These statements will be frequently used in the next sections. 

\begin{lemma}\label{lem:AlgoProj}
Let $X$ be a geodesically constructible quasi-median graph. There exists an algorithm that determines, given a vertex $x \in X$ and an edge $e \subset X$, the projection of $x$ on the clique containing $e$.
\end{lemma}

\begin{proof}
Let $C$ denote the clique containing $e$ and $p,q$ the endpoints of $e$. The projection of $x$ on $C$ is the unique vertex in $I(x,p) \cap I(x,q) \cap C$. By assumption, $I(x,p)$ and $I(x,q)$ are algorithmically constructible. And a vertex belongs to $C$ if and only if it lies at distance $\leq 1$ from both $p$ and $q$. 
\end{proof}

\begin{cor}\label{cor:AlgoHyp}
Let $X$ be a geodesically constructible quasi-median graph. There exists an algorithm that determines whether two given edges in $X$ belong to the same hyperplane.
\end{cor}

\begin{proof}
Let $e,e' \subset X$ be two edges. As a a consequence of Proposition~\ref{prop:Projection}, $e$ belongs to the same hyperplane as $e'$ if and only if the endpoints of $e$ have distinct projections on the clique containing $e'$. Thus, our corollary follows from Lemma~\ref{lem:AlgoProj}. 
\end{proof}

\begin{lemma}\label{lem:AlgoTransverse}
Let $X$ be a geodesically constructible quasi-median graph. There exists an algorithm that determines, given two edges $e,e' \subset X$, whether or not the hyperplanes containing $e$ and $e'$ are transverse.
\end{lemma}

\begin{proof}
The key observation is that $J$ and $J'$ are transverse if and only if there exists a geodesic between two endpoints of $e$ and $e'$ with two consecutive edges spanning a $4$-cyclic and belonging to $J$ and $J'$. 

\medskip \noindent
Indeed, let $p$ and $p'$ be two endpoints of $e$ and $e'$ that are separated by both $J$ and $J'$. Fix a geodesic $\gamma_0$ from $p$ to $p'$ and let $\epsilon,\epsilon' \subset \gamma_0$ denote the two edges in $J,J'$. If $J$ and $J'$ are transverse, then every hyperplane crossing $\gamma_0$ between $\epsilon$ and $\epsilon'$ are transverse to $J$ or $J'$. If $\epsilon$ and $\epsilon'$ are adjacent, we are done. Otherwise, let $\eta \subset \gamma_0$ denote the closest edge from $\epsilon$ that lies between $\epsilon$ and $\epsilon'$ and whose hyperplane $J(\eta)$ is transverse to $J$. If $\epsilon$ and $\eta$ are not adjacent, then the hyperplane containing the edge $\eta'$ between $\epsilon$ and $\eta$ adjacent to $\eta$ must be transverse to $J(\eta)$, so $\eta$ and $\eta'$ span a $4$-cycle and we can modify $\gamma_0$ by flipping this $4$-cycle in order to decrease the distance between $\eta$ and $\epsilon$. Thus, up to modifying the geodesic $\gamma_0$, we can assume that $\epsilon$ and $\eta$ are adjacent. Necessarily, these two edges span a $4$-cycle, because their hyperplanes are transverse, and we can again modify $\gamma_0$ by flipping this $4$-cycle in order to decrease the distance between $\epsilon$ and $\epsilon'$. After a few iterations, we get a geodesic with $\epsilon$ and $\epsilon'$ adjacent, providing the desired geodesic. The converse of our claim is clear.

\medskip \noindent
Because geodesics are constructible, that spanning a $4$-cycle is algorithmically decidable as a particular, and that belonging to a given hyperplane is algorithmically decidable according to Corollary~\ref{cor:AlgoHyp}, our lemma follows from the characterisation just proved. 
\end{proof}

\begin{lemma}\label{lem:AlgoCarrier}
Let $X$ be a constructible quasi-median graph. There exists an algorithm that decides, given a vertex $x \in X$ and an edge $e \subset X$, whether or not $x$ belongs to the carrier of the hyperplane containing $e$.
\end{lemma}

\begin{proof}
Because $X$ is locally constructible, we get algorithmically a collection of edges $e_1, \ldots, e_n$ containing $x$ such that every hyperplane containing $x$ in its carrier contains one of the $e_i$. So just need to check whether $e$ belongs to the same hyperplane as one of the $e_i$, which is possible algorithmically according to Corollary~\ref{cor:AlgoHyp}.
\end{proof}

\begin{lemma}\label{lem:AlgoMinDist}
Let $X$ be a constructible quasi-median graph. There exists an algorithm that gives, given two edges $e,e' \subset X$, two vertices $x,x' \in X$ minimising the distance between the carriers of the hyperplanes containing $e,e'$.
\end{lemma}

\begin{proof}
Let $N$ and $N'$ denote the carriers of the hyperplanes containing $e$ and $e'$. Fix two endpoints $p$ and $p'$ of $e$ and $e'$. Let $q$ denote the projection of $p$ on $N'$ and $q'$ the projection of $q$ on $N$. We know that $q$ belongs to $I(p,p')$ and that $q'$ belongs to $I(q,p)$, so there exists a geodesic between $p$ and $p'$ passing through $q$ and $q'$. But it follows from Lemma~\ref{lem:MinDistance} that $p$ and $p'$ are two vertices minimising the distance between $N$ and $N'$. 

\medskip \noindent
Therefore, in order to get two vertices minimising the distance between $N$ and $N'$, it suffices, for every geodesic $\gamma$ between $p$ and $p'$, to find the last vertex $a_\gamma$ in $\gamma$ belonging to $N$ and the first vertex $a_\gamma'$ in $N'$ (algorithmically thanks to Lemma~\ref{lem:AlgoCarrier}), and to choose a pair $a_\gamma,a_\gamma'$ with $d(a_\gamma,a_\gamma')$ minimal. Because geodesics are constructible, the whole process is algorithmic. 
\end{proof}

\begin{cor}\label{cor:AlgoContact}
Let $X$ be a constructible quasi-median graph. There exists an algorithm that determines, given two edges $e,e' \subset X$, whether or not the hyperplanes containing $e$ and $e'$ are in contact.
\end{cor}

\begin{proof}
Two hyperplanes are in contact if and only if the distance between their carriers is zero, so the desired conclusion follows from Lemma~\ref{lem:AlgoMinDist}.
\end{proof}

\begin{lemma}\label{lem:AlgoStronglySeparated}
Let $X$ be a constructible quasi-median graph. There exists an algorithm that determines, given two edges $e,e' \subset X$, whether or not the hyperplanes containing $e$ and $e'$ are strongly separated.
\end{lemma}

\begin{proof}
Let $J$ and $J'$ denote the hyperplanes containing $e$ and $e'$, and $N$ and $N'$ their carriers. According to Lemma~\ref{lem:AlgoMinDist}, we can find algorithmically two vertices $x \in N$ and $x' \in N'$ minimising the distance between $N$ and $N'$. If $J$ and $J'$ are not strongly separated, then the projection of $N'$ on $N$ is not reduced to single vertex, so $x$ must have a neighbour such that the hyperplane $H$ separating these two vertices is transverse t both $J$ and $J'$. But $X$ is locally constructible, so we can get algorithmically edges $e_1, \ldots, e_n$ containing $x$ such that $H$ contains one of the $e_i$ and we can check thanks to Lemma~\ref{lem:AlgoTransverse} that $H$ is transverse to both $J$ and $J'$. Thus, the existence of a hyperplane transverse to both $J$ and $J'$ is decidable algorithmically. 
\end{proof}

\subsection{Strongly contracting isometries}

\noindent
This section is dedicated to the proof of the following theorem. As a consequence of Proposition~\ref{prop:Loxo}, it can be thought of as a particular case of Theorem~\ref{thm:AsymptoticLength} since it recognises algorithmically the isometries with zero translation length in the crossing and contact graphs.

\begin{thm}\label{thm:WhenStronglyContracting}
Let $X$ be a constructible quasi-median graph. Let $\Omega X$ denote the crossing or contact graph of $X$. In the former case, we assume that $X$ has no cut-vertex; and, in any case, we assume that every vertex of $X$ belongs to $\leq N$ cliques. There exists an algorithm that determines, given a computable isometry $g \in \mathrm{Isom}(X)$, whether or not $g$ is strongly contracting.
\end{thm}

\noindent
Our argument has two steps. First, we show that crossing and contact graphs of constructible quasi-median graphs inherit weak algorithmic properties. More precisely, Proposition~\ref{prop:AlgoDistCrossing} below shows geodesics in crossing and contact graphs can be algorithmically constructed in the following sense:

\begin{definition}
Let $X$ be a graph. A \emph{combing} is a map $\gamma$ that associates to every pair of vertices $(x,y) \in X^2$ a path $\gamma(x,y)$ from $x$ to $y$. It is \emph{geodesic} if $\gamma(x,y)$ is a geodesic for all $x,y \in X$, and \emph{constructible} if there exists an algorithm that, given two vertices $x,y \in X$, yields $\gamma(x,y)$. 
\end{definition}

\noindent
Our second step is to show that, in a hyperbolic graph admitting a constructible geodesic combing, it can be decided whether a (non-parabolic) isometry is elliptic or loxodromic. This is Proposition~\ref{prop:EllOrLoxo}. 

\medskip \noindent
Before turning to the proofs of our proposition, let us mention that, from the algorithmic point of view, we think about crossing and contact graphs slightly differently as before. Indeed, a vertex in such a graph is, by definition, a hyperplane. But a hyperplane, as an infinite collection of edges, is not algorithmic-friendly. Instead, we think of a vertex in the crossing or contact graph as an edge of the quasi-median graph under consideration, and we define two edges as representing the same vertex if they belong to the same hyperplane (which can be decided algorithmically according to Lemma~\ref{lem:AlgoTransverse}); two edges represent two adjacent vertices if the hyperplanes containing them are transverse or in contact (which can be decided algorithmically according to Corollary~\ref{cor:AlgoContact}). Formally, this amounts to working with a pseudo-metric on the set of the edges of our quasi-median graph. 

\begin{prop}\label{prop:AlgoDistCrossing}
Let $X$ be a constructible quasi-median graph. Let $\Omega X$ denote the crossing or contact graph of $X$; in the former case, we assume that $X$ has no cut-vertex. Then $\Omega X$ admits a constructible geodesic combing.
\end{prop}

\noindent
The proposition will essentially follow from the combination of our next lemma with the basic results of the previous section.

\begin{lemma}\label{lem:CrossingFellowTravel}
Let $X$ be a quasi-median graph, $A,B$ two hyperplanes, and $a \in N(A), b \in N(B)$ two vertices. There exist a geodesic $H_1, \ldots, H_r$ between $A$ and $B$ in the crossing (resp. contact) graph and a geodesic $\gamma$ between $a$ and $b$ in $X$ such that $\gamma \cap N(H_i) \neq \emptyset$ for every $1 \leq i \leq r$. 
\end{lemma}

\begin{proof}
Let $H_0, \ldots, H_r$ be a geodesic in the crossing (resp. contact) graph such that the quantity $d(a,N(H_0))+ \cdots + d(a,N(H_r))$ is as small as possible. Define the sequence $(x_i)_{0 \leq i \leq r+1}$ as follows:
\begin{itemize}
	\item $x_0:=a$ and $x_{r+1}:=b$;
	\item for every $1 \leq i \leq r$, $x_i$ is the projection of $x_{i-1}$ on $N(H_i)$.
\end{itemize}
For every $0 \leq i \leq r$, fix a geodesic $[x_i,x_{i+1}]$ between $x_i$ and $x_{i+1}$, and let $\gamma$ denote the concatenation $[x_0, x_1] \cup \cdots \cup [x_r, x_{r+1}]$. By construction, $x_i \in \gamma \cap N(H_i)$ for every $1 \leq i \leq r$. In order to conclude our lemma, it suffices to show that $\gamma$ is a geodesic.

\medskip \noindent
If $\gamma$ is not a geodesic, then it crosses twice some hyperplane. Observe that, for every $0 \leq i \leq r-1$, no hyperplane crosses both $[x_i,x_{i+1}]$ and $[x_{i+1},x_{i+2}]$ because every hyperplane separating $x_i$ from its projection $x_{i+1}$ has to separate $x_i$ from $N(H_{i+1})$. Also, a hyperplane $J$ cannot cross $[x_i,x_{i+1}]$ and $[x_j,x_{j+1}]$ for some $0 \leq i,j \leq r-1$ satisfying $|i-j|>2$ since otherwise we could shorten $H_0, \ldots, H_r$ by replacing $H_i, \ldots, H_j$ with $H_i,J,H_j$. Therefore, if $\gamma$ is not a geodesic, there must exist some $0 \leq i \leq r-2$ such that $[x_i,x_{i+1}]$ and $[x_{i+2},x_{i+3}]$ are crossed by a common hyperplane $J$. As a consequence of what we already know, $J$ cannot cross $\gamma$ a third time, so it separates $a$ from $N(H_{i+1})$. Because $J$ must be transverse to both $H_i$ and $H_{i+2}$, $H_0, \ldots, H_i,J,H_{i+2}, \ldots, H_r$ is also a geodesic in the crossing (resp. contact) graph of $X$ with
$$d(a,N(H_0)) + \cdots d(a,N(H_i))+d(a,N(J))+d(a,N(H_{i+2}))+ \cdots d(a,N(H_r))$$
$$< d(a,N(H_0)) + \cdots + d(a,N(H_r)),$$
contradicting the minimality satisfied by $H_0, \ldots, H_r$. 
\end{proof}

\begin{proof}[Proof of Proposition~\ref{prop:AlgoDistCrossing}.]
Let $e,e' \subset X$ be two edges. Fix to endpoints $p$ and $p'$ of $e$ and $e'$. For all geodesic $\gamma$ from $p$ to $p'$ (which can be enumerated algorithmically because $X$ is geodesically constructible), integer $1 \leq k \leq d(p,p')$, vertices $x_0:=p,x_1, \ldots, x_{k-1}, x_k:=p'$ of $\gamma$, and hyperplanes $J_1, \ldots,J_{k-1}$ containing respectively $x_1, \ldots, x_{k-1}$ in their carriers (which can be enumerated algorithmically because $X$ is locally constructible), check whether $A,J_1, \ldots, J_{k-1}, B$ defines a path in the crossing (resp. contact) graph (which can be done algorithmically according to Lemma~\ref{lem:AlgoTransverse} (resp. Corollary~\ref{cor:AlgoContact})). As a consequence of Lemma~\ref{lem:CrossingFellowTravel}, the path in the crossing (resp. contact) graph with the smallest $k$ yields a geodesic between $A$ and $B$ in the crossing (resp. contact) graph of $X$. We choose such a geodesic in order to define our constructible geodesic combing.
\end{proof}

\noindent
Next, we turn to the second step of the proof of our theorem.

\begin{prop}\label{prop:EllOrLoxo}
Let $X$ be a $\delta$-hyperbolic graph admitting a constructible geodesic combing. There exists an algorithm that determines, given a non-parabolic computable isometry $g \in \mathrm{Isom}(X)$, whether $g$ is elliptic or loxodromic.
\end{prop}

\noindent
Our proposition follows from our next two lemmas, which allow us to distinguish elliptic and loxodromic isometries in hyperbolic graphs.

\begin{lemma}\label{lem:WhenLoxo}
Let $X$ be a $\delta$-hyperbolic graph and $g \in \mathrm{Isom}(X)$ a loxodromic isometry. For every $x \in X$ and every $k \in \mathbb{Z}$, $d(x,g^kx) \geq |k| +16 \delta$.
\end{lemma}

\begin{proof}
According to \cite[Proposition~10.6.4]{MR1075994}, we have
$$d(x,g^kx) \geq \left[ g^k \right] \geq \left\| g^k \right\| +16 \delta = |k| \cdot \|g \| +16 \delta,$$
where, for every isometry $h$, $[h]:= \min \{ d(z,hz) \mid z \in X\}$ and $\|h\|:= \lim_{n \to + \infty} d(o,h^no)/n$ for an arbitrary basepoint $o \in X$. But $X$ is a graph, so we must have $\|g\| \geq 1$, concluding the proof of our lemma. 
\end{proof}

\begin{lemma}\label{lem:WhenEll}
Let $X$ be a $\delta$-hyperbolic graph and $g \in \mathrm{Isom}(X)$ an elliptic isometry. For every vertex $x \in X$, every integer $k \in \mathbb{Z}$, and every geodesic $[x,g^kx]$ between $x$ and $g^kx$, there exists a vertex $y \in [x,g^kx]$ satisfying $d(y,g^ky) \leq 32\delta$.
\end{lemma}

\begin{proof}
We begin by proving the following observation:

\begin{fact}
The subgraph $F:= \{ z \in X \mid \mathrm{diam}( \langle g \rangle \cdot z \leq 2 \delta\}$ is non-empty and $12\delta$-quasiconvex.
\end{fact}

\noindent
The fact that $F$ is non-empty follows from \cite{MR1776048}. To be precise, given a bounded subgraph $B \subset X$, define its radius $\mathrm{rad}(B)$ as $\min \{ r \geq 0 \mid \exists c \in X, S \subset B(c,r) \}$. A vertex $c \in X$ is a centre of $B$ if $B \subset B(c, \mathrm{rad}(B))$. According to \cite[Lemma~2.1]{MR1776048}, the set of the centres of a bounded subgraph has diameter at most $2\delta$, so $F$ contains the centre of every $\langle g \rangle$-orbit. (It is worth mentioning that \cite[Lemma~2.1]{MR1776048} assumes that the hyperbolic space is proper, but only in order to assure the existence of a centre, which is not needed here since our metric is discrete.) Now, let $a,b \in F$ be two vertices and $z \in I(a,b)$ a third vertex. Because the distance in $X$ is $8\delta$-convex \cite[Corollaire~10.5.3]{MR1075994}, we know that $\langle g \rangle \cdot z$ has diameter at most $12 \delta$. Taking a centre $c$ of this orbit, we have 
$$d(z,F) \leq d(z,c) \leq \mathrm{rad}(\langle g \rangle  \cdot z) \leq \mathrm{diam}(\langle g \rangle \cdot z) \leq 12\delta,$$
concluding the proof of our fact.

\medskip \noindent
\begin{minipage}{0.4\linewidth}
\includegraphics[width=0.9\linewidth]{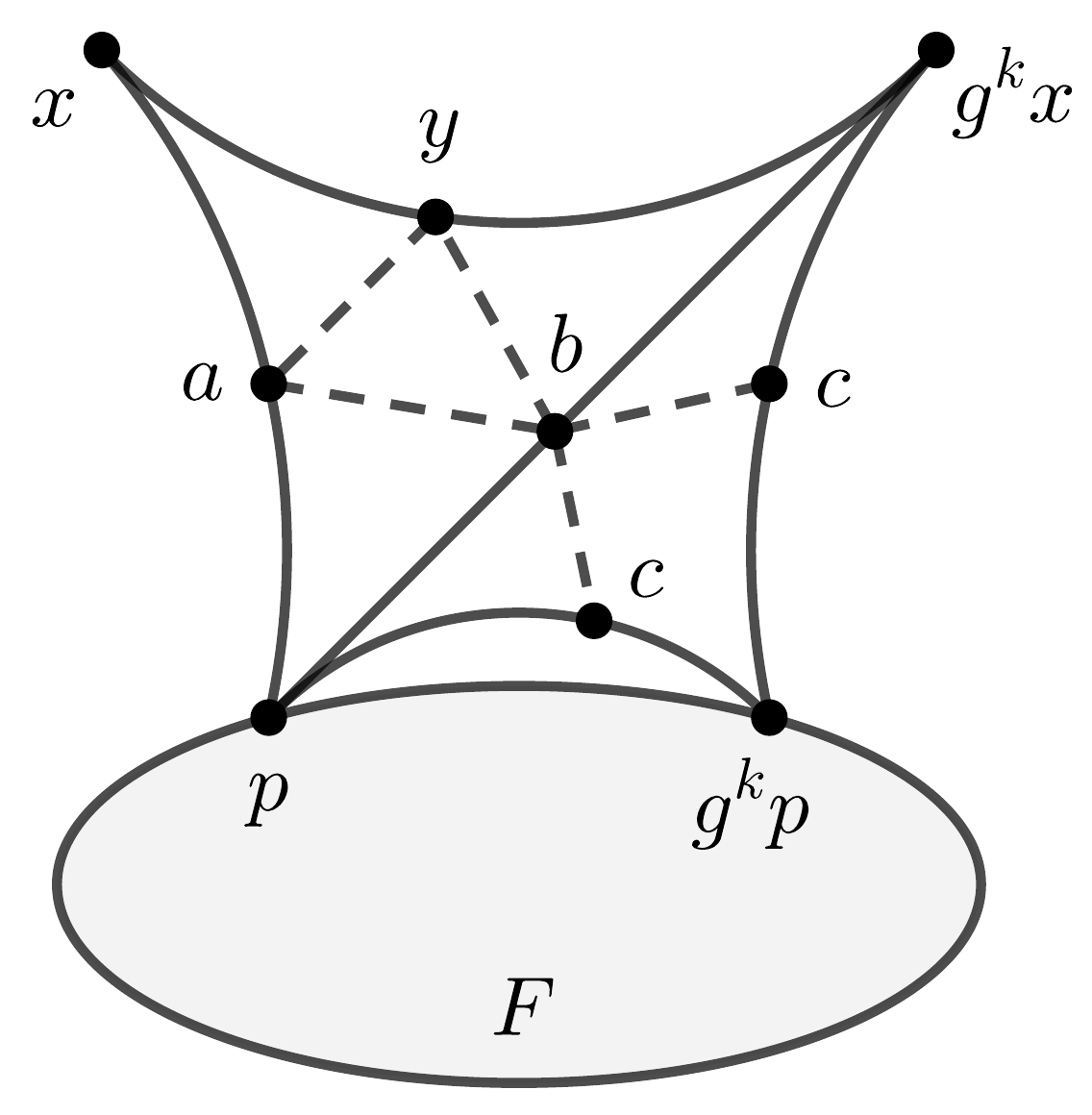}
\end{minipage}
\begin{minipage}{0.58\linewidth}
Now, fix a vertex $x \in X$, an integer $k \in \mathbb{Z}$, and a geodesic $[x,g^kx]$ between $x$ and $g^kx$. Also, fix an arbitrary vertex $p \in F$ and some geodesics $[x,p]$, $[p,g^kx]$, and $[p,g^kp]$. Because $X$ is $\delta$-hyperbolic, there exist vertices $y \in [x,g^kx]$, $a \in [x,p]$, and $b \in p,g^kx]$ pairwise at distance $\leq 2\delta$. Also, there exists $c \in [p,g^kp] \cup g^k[p,x]$ such that $d(b,c) \leq \delta$. 
\end{minipage}

\medskip \noindent
If $c \in [p,g^kp]$, then
$$d(y,F) \leq d(y,b)+d(b,c)+d(c,F) \leq 2\delta + \delta +12\delta= 15\delta.$$
Therefore, $\mathrm{diam}(\langle g \rangle \cdot y) \leq 32 \delta$.

\medskip \noindent
Otherwise, if $c \in g^k[p,x]$, we have
$$|d(g^kp,c) - d(p,a)] \leq d(a,c)+d(p,g^kp) \leq 3\delta +2\delta=5 \delta.$$
Because $c$ and $g^ka$ lies on the same geodesic $g^k[p,x]$, this implies that $d(g^ka,c) \leq 5\delta$. Thus,
$$d(g^ky,y) \leq d( g^ky,g^ka) + d(g^ka,c) + d(c,y) \leq 2\delta +5\delta +3\delta = 10 \delta.$$
Therefore, we have proved that $d(g^ky,y) \leq 32\delta$ in any case, concluding the proof of our lemma.
\end{proof}

\begin{proof}[Proof of Proposition~\ref{prop:EllOrLoxo}.]
Fix an arbitrary vertex $o \in X$. Given a computable isometry $g \in \mathrm{Isom}(X)$, construct the geodesic $[o,g^{17\delta}o]$ given by our combing, and, for every vertex $x \in [o,g^{17\delta}o]$, compute $d(x,g^{17\delta}x)$. As a consequence of Lemmas~\ref{lem:WhenLoxo} and~\ref{lem:WhenEll}, either all these values are greater than $33\delta$, and $g$ must be loxodromic; or $g$ is elliptic otherwise. 
\end{proof}

\begin{proof}[Proof of Theorem~\ref{thm:WhenStronglyContracting}.]
As a consequence of Propositions~\ref{prop:AlgoDistCrossing} and~\ref{prop:EllOrLoxo}, there exists an algorithm that determines, given a computable isometry $g \in \mathrm{Isom}(X)$ that does not include a parabolic isometry on the crossing (resp. contact) graph, whether or not $g$ is loxodromic in the crossing (resp. contact) graph, which amounts to saying that $g$ is strongly contracting according to Proposition~\ref{prop:Loxo}. Since we know from Proposition~\ref{prop:Loxo} that an isometry of $X$ cannot induce a parabolic isometry of the crossing (resp. contact) graph, our theorem follows. 
\end{proof}

\subsection{Computing translation lengths}

\noindent
We are now finally ready to prove the second main result of this article, namely Theorem~\ref{thm:AsymptoticLength}. In fact, we are going to prove the following stronger version of it:

\begin{thm}\label{thm:BigAlgo}
Let $X$ be a constructible quasi-median graph, and let $\Omega X$ be either the crossing graph or the contact graph of $X$. In the former case, we assume that $X$ has not cut-vertex; and, in any case, we assume that a vertex of $X$ belongs to at most $N$ cliques for some fixed $N \geq 1$. There exists an algorithm that computes, given a computable isometry $g \in \mathrm{Isom}(X)$, the translation length of $g$ in $\Omega X$. Moreover, if the translation length is positive, then the algorithm provides an integer $k \geq 1$ and a hyperplane $J$ such that $J$ belongs to an axis of $g^k$ in $\Omega X$. 
\end{thm}

\noindent
In addition of the preliminary work done so far, we need to be able to algorithmically find axes of strongly contracting isometries in constructible quasi-median graphs. 

\begin{prop}\label{prop:FindAxis}
Let $X$ be a constructible quasi-median graph. There exists an algorithm that provides, given a computable strongly contracting isometry $g \in \mathrm{Isom}(X)$, a vertex $x \in X$ and an integer $N \geq 1$ such that $x$ belongs to an axis of $g^N$. 
\end{prop}

\noindent
The following notion will be useful in order to find axes.

\begin{definition}
Let $X$ be a quasi-median graph. An isometry $g \in \mathrm{Isom}(X)$ \emph{skewers} an oriented edge $e:=(x,y)$ if the sector containing $y$ and delimited by the hyperplane containing $e$ satisfies $gS \subsetneq S$.
\end{definition}

\noindent
As the typical example to keep in mind, let $\gamma$ be an axis of our isometry $g$ and let $e \subset \gamma$ be an edge. Orient $\gamma$ such that $g$ translates the vertices $\gamma$ in the positive direction and endow $e$ with the induced orientation. Then, if the hyperplanes containing $e$ and $ge$ are not transverse, then $g$ skewers $e$. 

\medskip \noindent
First, we notice that skewering an edge can be detected algorithmically.

\begin{lemma}\label{lem:Skewers}
Let $X$ be a quasi-median graph, $g \in \mathrm{Isom}(X)$ an isometry, and $e:=(x,y)$ an oriented edge. Then $g$ skewers $e$ if and only if the following conditions hold:
\begin{itemize}
	\item $gy$ does not belong to $I(y,gx)$;
	\item the hyperplanes containing $e$ and $ge$ are distinct and non-transverse.
\end{itemize}
\end{lemma}

\begin{proof}
First, assume that $g$ skewers $e$. Let $S$ denote the sector containing $y$ delimited by the hyperplane $J$ containing $e$. Clearly, $gS \subsetneq S$ implies that the hyperplanes $J$ and $gJ$ are distinct and non-transverse. Moreover, $y$ and $gx$ do not belong to $gS$, hence $I(y,gx) \subset gS^c$. This implies that $gy$, which belongs to $gS$, cannot belong to $I(y,gx)$. Conversely, assume that the two conditions given by our lemma hold. Because $J$ and $gJ$ are distinct and non-transverse, we have $gS \subsetneq S$ unless $gS$ is the sector delimited by $gJ$ containing $J$. If so, $gy$ coincides with the projection of $y$ on the clique containing $gx$ and $gy$, hence $gy \in I(y,gx)$, contradicting our first condition. 
\end{proof}

\begin{cor}\label{cor:SkewersAlgo}
Let $X$ be a geodesically constructible quasi-median graph. There exists an algorithm that determines, given a computable isometry $g \in \mathrm{Isom}(X)$ and an oriented edge $e$, whether or not $g$ skewers $e$.
\end{cor}

\begin{proof}
The first condition of Lemma~\ref{lem:Skewers} is decidable algorithmically because $X$ is geodesically constructible; as well as the second condition as a consequence of Corollary~\ref{cor:AlgoHyp} and Lemma~\ref{lem:AlgoTransverse}. 
\end{proof}

\noindent
Next, we show that powers of strongly contracting isometries skewer edges.

\begin{lemma}\label{lem:FindSkewers}
Let $X$ be a quasi-median graph and $g \in \mathrm{Isom}(X)$ a strongly contracting isometry. There exists some $N \geq 1$ such that, for every vertex $x \in X$ and every geodesic $[x,g^Nx]$ from $x$ to $g^Nx$, some oriented edge $e \subset [x,g^Nx]$ is skewered by $g^N$ and the hyperplanes containing $e$ and $g^Ne$ are strongly separated.
\end{lemma}

\begin{proof}
According to Proposition~\ref{prop:Loxo}, there exists some integer $N \geq 1$ and some hyperplane $J$ crossing an axis of $g$ such that $\{g^{kN}J \mid k \in \mathbb{Z} \}$ is a collection of pairwise strongly separated hyperplanes. Up to replacing $J$ with a $\langle g^N \rangle$-translate, assume that our vertex $x$ lies between $g^{-N}J$ and $J$. Of course, $g^Nx$ lies between $J$ and $g^NJ$. Therefore, our geodesic $[x,g^Nx]$ from $x$ to $g^Nx$ must cross $J$; let $e$ denote the corresponding oriented edge. The hyperplanes containing $e$ and $g^Ne$, namely $J$ and $g^NJ$, are strongly separated; and $g^N$ clearly skewers $e$. 
\end{proof}

\noindent
Finally, we explain how skewered edges can be helpful in order to find axes of strongly contracting isometries. 

\begin{lemma}\label{lem:BelongsAxis}
Let $X$ be a quasi-median graph, $g \in \mathrm{Isom}(X)$ a loxodromic isometry, and $e \subset X$ an oriented edge skewered by $g$. If the hyperplanes containing $e$ and $ge$ are strongly separated, then the two vertices minimising the distance between the carriers of these hyperplanes belong to an axis of $g^2$.
\end{lemma}

\begin{proof}
Let $x,y \in X$ be two vertices such that $e=(x,y)$, let $J$ denote the hyperplane containing $e$, and let $S$ denote the sector delimited by $J$ containing $y$. Given an axis $\gamma$ of $g$, notice that $J$ crosses $\gamma$. Indeed, otherwise either $\gamma \subset S$, and $J,gJ,g^2J, \ldots$ yields infinitely many hyperplanes separating $S^c$ and $\gamma$, which is impossible; or $S \cap \gamma= \emptyset$, and, given an arbitrary vertex $v \in S$, the hyperplanes $J, \ldots, g^nJ$ separate $g^nv$ from $\gamma$ for every $n \geq 1$, which is also impossible since $d(g^nx, \gamma)=d(x,\gamma)$ for every $n \geq 1$. Thus, $\{ g^k J \mid k \in \mathbb{Z}\}$ defines a collection of pairwise strongly separated hyperplanes crossing the axis $\gamma$ of $g$. Fix a vertex $p \in \gamma$ lying between $g^{-1}J$ and $J$. Of course, $g^2p$ lies between $gJ$ and $g^2J$. As a consequence of Proposition~\ref{prop:Axis}, $g^2$ admits an axis passing through any fixed geodesic between $p$ and $g^2p$. But, given two vertices $x \in N(J)$ and $y \in N(gJ)$ minimising the distance between $N(J)$ and $N(gJ)$, concatenating a geodesic from $p$ to $x$ with a geodesic from $x$ to $y$ and next a geodesic from $y$ to $g^2p$ yields a geodesic from $p$ to $g^2p$. (Indeed, no hyperplane can cross both $[p,x]$ and $[y,g^2p]$ because $J$ and $gJ$ are strongly separated; and no hyperplane can cross both $[p,x]$ and $[x,y]$ or $[x,y]$ and $[y,g^2p]$ since the hyperplanes separating $x$ and $y$ coincide with the hyperplanes separating $N(J)$ and $N(gJ)$. So no hyperplane can cross our path twice, proving that it is a geodesic.) 
\end{proof}

\begin{proof}[Proof of Proposition~\ref{prop:FindAxis}.]
Fix an arbitrary vertex $o \in X$. For successive powers $k=1,2,\ldots$, check whether some oriented edge $e$ of some geodesic between $o$ and $g^ko$ satisfies the following condition: $g^k$ skewers $e$ and the hyperplanes containing $e$ and $g^ke$ are strongly separated. This can be done algorithmically according to Corollary~\ref{cor:SkewersAlgo} and Lemma~\ref{lem:AlgoStronglySeparated}. According to Lemma~\ref{lem:FindSkewers}, we eventually find a power $N \geq 1$ and an oriented edge $e$ for which it works. Applying Lemma~\ref{lem:AlgoMinDist} to the edges $e$ and $g^Ne$ yields two vertices that, according to Lemma~\ref{lem:BelongsAxis}, belong to an axis of $g^{2N}$. So our algorithm outputs one of these two vertices and the integer $2N$. 
\end{proof}

\begin{proof}[Proof of Theorem~\ref{thm:BigAlgo}.]
Let $g \in \mathrm{Isom}(X)$ be a computable isometry. 

\medskip \noindent
\emph{Step 1.} Determine whether or not $g$ is strongly contracting. If no, stop and output $\| g \|_{\Omega X}=0$. If yes, pass to the next step.

\medskip \noindent
Checking whether $g$ is strongly contracting can be done algorithmically according to Theorem~\ref{thm:WhenStronglyContracting}. If not, then it follows from Proposition~\ref{prop:Loxo} that $g$ has bounded orbits in $\Omega X$, hence $\| g \|_{\Omega X}=0$. From now on, we assume that $g$ is strongly contracting.

\medskip \noindent
\emph{Step 2.} Find an edge $e$ and a power $L \geq 1$ such that $e$ is contained in an axis of $g^L$ and the hyperplanes containing $e$ and $g^Le$ are strongly separated.

\medskip \noindent
According to Proposition~\ref{prop:FindAxis}, we can find algorithmically a vertex $x \in X$ and a power $n \geq 1$ such that $x$ belongs to an axis of $g^n$. Construct a geodesic $[x,g^nx]$. We know from Proposition~\ref{prop:Axis} that the $\langle g^n \rangle$-translates of $[x,g^nx]$ yields an axis $\gamma$ of $g^n$. Testing the successive values of $k \geq 0$, find an edge $e \subset [x,g^nx]$ such that the hyperplanes containing $e$ and $g^{kn}e$ are strongly separated. This can be done algorithmically according to Lemma~\ref{lem:AlgoStronglySeparated} and we know that we will eventually find such a $k$ thanks to Proposition~\ref{prop:Loxo}. Then $e$ and $L:=kn$ are the data we are looking for. 

\medskip \noindent
\emph{Step 3.} Compute $\mathrm{HQC}(g)$. 

\medskip \noindent
Let $\mathcal{H},\mathcal{V}$ be two finite transverse collection of hyperplanes crossing the axes of $g$. Up to translating $\mathcal{H},\mathcal{V}$ by a power of $g^L$, we assume that at least one hyperplane $H$ in $\mathcal{H} \cup \mathcal{V}$ lies between $g^{-L}J$ and $g^LJ$, where $J$ denotes the hyperplane containing our edge $e$. Without loss of generality, say that $H$ belongs to $\mathcal{H}$. Because every hyperplane in $\mathcal{V}$ is transverse to $H$, they must lie between $g^{-2L}J$ and $g^{2L}J$. And because every hyperplane in $\mathcal{H}$ is transverse to every hyperplane in $\mathcal{V}$, they must lie between $g^{-3L}$ and $g^{3L}$. Thus, every hyperplane in $\mathcal{H} \cup \mathcal{V}$ must separate $g^{-3L}e$ and $g^{3L}e$. We immediately deduce that $\mathrm{HQC}(g) \leq 6L \|g\|_X$, which is sufficient for the rest of the proof; but we can also compute precisely $\mathrm{HQC}(g)$ by studying the pattern of transverse hyperplanes separating $g^{-3L}e$ and $g^{3L}e$ (Lemma~\ref{lem:AlgoTransverse}). 

\medskip \noindent
\emph{Step 4.} Let $x$ be an endpoint of $e$. Enumerate the hyperplanes in contact with 
$$B \left( x,N^{2 \mathrm{HQC}(g)} \right) \cap I \left( g^{-1-L(4 \mathrm{HQC}(g)+1)}x, g^{1+L(4 \mathrm{HQC}(g)+1)}x \right) $$ 
and, for each such hyperplane $H$, compute $d_{\Omega X} (H,g^{N^{2\mathrm{HQC}(g)}!}H)$. Pick a hyperplane $K$ for which this quantity is minimal. Then $K$ belongs to an axis of $g^{N^{2 \mathrm{HQC}(g)}!}$ in $\Omega X$ and $\|g\|_{\Omega X}= d_{\Omega X}(K, g^{N^{2 \mathrm{HQC}(g)}!} K)/ N^{2 \mathrm{HQC}(g)}!$. 

\medskip \noindent
We deduce from Theorem~\ref{thm:Rational} and Lemma~\ref{lem:FormulaQC} that there exists some hyperplane $H$ in contact with $F:=B(x,4 \mathrm{HQC}(g)+1) \cap \mathrm{CH}(\gamma)$ such that $K$ belongs to an axis of $g^k$ in $\Omega X$ where $k := N^{2 \mathrm{HQC}(g)}!$. 

\medskip \noindent
Observe that $F \subset I(g^{-Q}x,g^{Q}x)$ where $Q:=1+ L(4 \mathrm{HQC}(g)+1)$. Indeed, $x$ lies between $g^{-L}J$ and $g^LJ$, so $B(x,4 \mathrm{HQC}(g)+1)$ must lie between $g^{-Q+1}J$ and $g^{Q-1}J$. Now, if $p$ belongs to $\mathrm{CH}(\gamma)$ and lies between $g^{-Q+1}J$ and $g^{Q-1}J$, then $p$ must lie in $I(g^{-Q}x,g^Qx)$ because otherwise there would exist some hyperplane separating $p$ from $\{g^{-Q}x,g^Qx\}$ but also crossing $\gamma$, necessarily not between $g^{-Q}x$ and $g^Qx$, so such a hyperplane would be transverse to both $g^{-Q+1}J$ and $g^{-Q}J$ or to both $g^{Q-1}J$ and $g^QJ$, which is impossible since these two pairs are strongly separated. 

\medskip \noindent
Thus, there exists some hyperplane $H$ in contact with $F':= B(x,4\mathrm{HQC}(g)+1) \cap I(g^{-Q}x,g^{Q}x)$ such that $H$ belongs to an axis of $g^k$ in $\Omega X$. We can enumerate the hyperplanes in contact with $F'$ because an interval is always finite (Lemma~\ref{lem:IntervalFinite}) and because $X$ is locally constructible. Because distances in $\Omega X$ are computable according to Proposition~\ref{prop:AlgoDistCrossing}, we can pick a hyperplane $K$ in our collection for which $d_{\Omega X}(K, g^kK)$ is minimal. It follows from Lemma~\ref{lem:AxisFromAnother} that $K$ belongs to an axis of $g^k$. We have
$$\|g\|_{\Omega X} = \frac{1}{k} \| g^k \| = \frac{1}{k} \lim\limits_{n \to + \infty} \frac{d_{\Omega X}(K,g^{kn}K)}{n} = d_{\Omega X}(K,g^k K)/k,$$
concluding the proof of our theorem. 
\end{proof}

\section{Graph products of groups}

\noindent
Let $\Gamma$ be a graph and $\mathcal{G}= \{ G_u \mid u \in V(\Gamma) \}$ a collection of groups indexed by the vertex-set $V(\Gamma)$ of $\Gamma$. The \emph{graph product} $\Gamma \mathcal{G}$ is 
$$\langle G_u \ (u\in V(\Gamma)) \mid [G_u,G_v]=1 \ (\{u,v\} \in E(\Gamma)) \rangle$$
where $E(\Gamma)$ denotes the edge-set of $\Gamma$ and where $[G_u,G_v]=1$ is a shorthand for $[g,h]=1$ for all $g \in G_u$, $h \in G_v$. The groups of $\mathcal{G}$ are referred to as \emph{vertex-groups}. We emphasize that, unless explicitly stated, vertex-groups are not assumed to be finite.

\medskip \noindent
\textbf{Convention.} We always assume that the groups in $\mathcal{G}$ are non-trivial. Notice that it is not a restrictive assumption, since a graph product with some trivial factors can be described as a graph product over a smaller graph all of whose factors are non-trivial.

\medskip \noindent
A \emph{word} in $\Gamma \mathcal{G}$ is a product $g_1 \cdots g_n$ where $n \geq 0$ and where, for every $1 \leq i \leq n$, $g_i \in G$ for some $G \in \mathcal{G}$; the $g_i$'s are the \emph{syllables} of the word, and $n$ is the \emph{length} of the word. Clearly, the following operations on a word does not modify the element of $\Gamma \mathcal{G}$ it represents:
\begin{description}
	\item[Cancellation:] delete the syllable $g_i$ if $g_i=1$;
	\item[Amalgamation:] if $g_i,g_{i+1} \in G$ for some $G \in \mathcal{G}$, replace the two syllables $g_i$ and $g_{i+1}$ by the single syllable $g_ig_{i+1} \in G$;
	\item[Shuffling:] if $g_i$ and $g_{i+1}$ belong to two adjacent vertex-groups, switch them.
\end{description}
A word is \emph{graphically reduced} if its length cannot be shortened by applying these elementary moves. Every element of $\Gamma \mathcal{G}$ can be represented by a graphically reduced word, and this word is unique up to the shuffling operation. 
For more information on graphically reduced words, we refer to \cite{GreenGP} (see also \cite{HsuWise,VanKampenGP}). 

\medskip \noindent
The connection between graph products and quasi-median graphs is made explicit by the following statement \cite[Proposition~8.2, Corollary~8.7]{QM}:

\begin{thm}
Let $\Gamma$ be a graph and $\mathcal{G}$ a collection of groups indexed by $V(\Gamma)$. The Cayley graph 
$$\mathrm{QM}(\Gamma, \mathcal{G}):= \mathrm{Cayl} \left( \Gamma \mathcal{G}, \bigcup\limits_{G \in \mathcal{G}} G \backslash \{1 \} \right)$$ 
is a quasi-median graph of cubical dimension $\mathrm{clique}(\Gamma)= \max \{ \# V(\Lambda ) \mid \Lambda \subset \Gamma \ \text{clique} \}$. 
\end{thm}

\noindent
Notice that $\Gamma \mathcal{G}$ naturally acts by isometries on $\mathrm{QM}(\Gamma, \mathcal{G})$ by left-multiplication and that, as a Cayley graph, the edges of $X(\Gamma, \mathcal{G})$ are naturally labelled by generators, but also by vertices of $\Gamma$ (corresponding to the vertex-group which contains the generator). It turns out that any two edges of $\mathrm{QM}(\Gamma, \mathcal{G})$ must be labelled by the same vertex of $\Gamma$ if they belong to the same hyperplane (see \cite[Lemma~8.9]{QM}), which implies that the hyperplanes of $\mathrm{QM}(\Gamma, \mathcal{G})$ are also naturally labelled by vertices of $\Gamma$. An easy observation that will be needed later is \cite[Lemma~8.12]{QM}, namely:

\begin{lemma}\label{lem:LabelHyp}
Let $\Gamma$ be a simplicial graph and $\mathcal{G}$ a collection of groups indexed by $V(\Gamma)$. Two transverse hyperplanes are labelled by adjacent vertices of $\Gamma$, and two tangent hyperplanes are labelled by distinct vertices of $\Gamma$. 
\end{lemma}

\noindent
Essentially by construction of the quasi-median graph, we have the following description of its geodesics \cite[Lemma~8.3]{QM}:

\begin{lemma}\label{lem:DistInX}
Let $\Gamma$ be a graph and $\mathcal{G}$ be a collection of groups indexed by $V(\Gamma)$. Fix two elements $g,h \in \Gamma \mathcal{G}$ and write $g^{-1}h$ as a graphically reduced word $u_1 \cdots u_n$. Then the sequence of vertices $$g,gu_1,gu_1u_2, \ldots, gu_1 \cdots u_n=h$$ defines a geodesic between $g$ and $h$ in $\mathrm{QM}(\Gamma, \mathcal{G})$. Conversely, any geodesic between $g$ and $h$ is labelled by a graphically reduced word representing $g^{-1}h$.
\end{lemma}

\noindent
Before turning to the proof of Theorem~\ref{thm:IntroGP}, we need a last preliminary lemma. Namely, we need to understand when the quasi-median graph of a graph product is hyperbolic. A characterisation is already given by \cite[Fact~8.33]{QM}, but we include a proof here with an estimate on the hyperbolicity constant. 

\begin{lemma}\label{lem:QMHyp}
Let $\Gamma$ be a simplicial graph and $\mathcal{G}$ a collection of groups indexed by $V(\Gamma)$. The quasi-median graph $\mathrm{QM}(\Gamma,\mathcal{G})$ is hyperbolic if and only if $\mathrm{clique}(\Gamma)$ is finite and $\Gamma$ has no induced $4$-cycle. If so, the graph is $5 \mathrm{clique}(\Gamma)$-hyperbolic. 
\end{lemma}

\begin{proof}
If $\Gamma$ contains a clique with $n$ vertices, then $\mathrm{QM}(\Gamma,\mathcal{G})$ contains an isometrically embedded product of $n$ complete graphs, and a fortiori of $[0,1]^n$. Therefore, if $\mathrm{clique}(\Gamma)$ is infinite, i.e.\ if $n$ can be chosen arbitrarily large, then $\mathrm{QM}(\Gamma, \mathcal{G})$ cannot be hyperbolic. If $\Gamma$ contains an induced $4$-cycle $(a,b,c,d)$, then, fixing non-trivial elements $p \in G_a$, $q \in G_b$, $r \in G_c$, and $s \in G_d$, the map $(i,j) \mapsto (pr)^i(qs)^j$ defines an isometric embedding $\mathbb{Z}^2 \to \mathrm{QM}(\Gamma,\mathcal{G})$. Therefore, $\mathrm{QM}(\Gamma,\mathcal{G})$ cannot be hyperbolic.

\medskip \noindent
Conversely, assume that $\Gamma$ has a finite clique number and has no induced $4$-cycle. Under these assumptions, observe that:

\begin{fact}\label{fact:Bigon}
The Hausdorff distance between two geodesics with the same endpoints is at most $2 \mathrm{clique}(\Gamma)$. 
\end{fact}

\noindent
Let $\alpha,\beta$ be two geodesics with the same endpoints and $x \in \alpha$ a vertex. It follows from \cite[Lemma~2.114]{QM}) that there exists an isometrically embedded copy of $[0,a] \times [0,a]$ in $\mathrm{QM}(\Gamma,\mathcal{G})$ with $(0,0)=x$ and $(a,a) \in \beta$. If $a>\mathrm{clique}(\Gamma)$, then there are at least two non-transverse hyperplanes separating $(0,0)$ and $(0,a)$. Because the hyperplanes in a maximal collection of pairwise non-transverse hyperplanes separating two vertices are successively tangent, it follows that there exist two tangent hyperplanes separating $(0,0)$ and $(0,a)$. Similarly, there must exist two tangent hyperplanes separating $(0,0)$ and $(a,0)$. Because these two pairs of hyperplanes are transverse, we deduce from Lemma~\ref{lem:LabelHyp} that $\Gamma$ contains an induced $4$-cycle, a contradiction. Thus, we must have $a \leq \mathrm{clique}(\Gamma)$, which implies that $d(p,\beta) \leq 2 \mathrm{clique}(\Gamma)$, and which finally prove our fact.

\medskip \noindent
Now, fix three vertices $a,b,c$ and three geodesics $[a,b]$, $[b,c]$, $[a,c]$. Let $(x,y,z)$ be the median triangle of $(a,b,c)$ and let $[a,x]$, $[b,y]$, $[c,z]$, $[x,y]$, $[y,z]$, $[x,z]$ be geodesics. Given a vertex $p \in [a,b]$, it follows from Fact~\ref{fact:Bigon} that there exists some $q \in [a,x] \cup [x,y] \cup [y,b]$ such that $d(p,q) \leq 2 \mathrm{clique}(\Gamma)$. If $q \in [a,x] \cup [y,b]$, then we can apply Fact~\ref{fact:Bigon} again and deduce that $p$ lies in the $4 \mathrm{clique}(\Gamma)$-neighbourhood of $[b,c] \cup [c,a]$. Otherwise, if $q \in [x,y]$, then it follows from Proposition~\ref{prop:MedianTriangle} and Lemma~\ref{lem:LabelHyp} that $d(q,y) \leq \mathrm{clique}(\Gamma)$. Applying Fact~\ref{fact:Bigon} again, we conclude that $p$ lies in the $5\mathrm{clique}(\Gamma)$-neighbourhood of $[b,c] \cup [c,a]$. 
\end{proof}

\begin{proof}[Proof of Theorem~\ref{thm:IntroGP}.]
The first assertion follows from Corollary~\ref{cor:IntroRational}, from the fact that every vertex in $\mathrm{QM}(\Gamma, \mathcal{G})$ belongs to exactly $|V(\Gamma)|$ cliques, and from the observation that $\mathrm{QM}(\Gamma, \mathcal{G})$ has no cut-vertex as justified by Claim~\ref{claim:CutVertex} and by the fact that the clique-link of a vertex in $\mathrm{QM}(\Gamma, \mathcal{G})$ is isomorphic to $\Gamma$. 

\medskip \noindent
The second assertion follows from Corollary~\ref{cor:HyperbolicCase} and Lemma~\ref{lem:QMHyp}. The third assertion follows from Theorem~\ref{thm:AsymptoticLength} and our next observation:

\begin{claim}
If the groups in $\mathcal{G}$ have solvable word problems, then $\mathrm{QM}(\Gamma, \mathcal{G})$ is constructible. 
\end{claim}

\noindent
Algorithmically, we think of $\mathrm{QM}(\Gamma, \mathcal{G})$ whose vertices are words in $\Gamma \mathcal{G}$, two words representing the same vertex if they are equal in $\Gamma \mathcal{G}$ (or equivalently, if they admit identical graphically reduced representatives, which can be checked algorithmically if the groups in $\mathcal{G}$ have solvable word problems); and whose edges connect two words if one can be obtained from the other by right-multiplying by a non-trivial element in some vertex-group. The fact that $\mathrm{QM}(\Gamma, \mathcal{G})$ is geodesically constructible follows from Lemma~\ref{lem:DistInX} (and from the fact that graphically reduced representatives of an element of $\Gamma \mathcal{G}$ can be easily enumerated because the groups in $\mathcal{G}$ have solvable word problems). In order to show that $\mathrm{QM}(\Gamma, \mathcal{G})$ is also locally constructible, fix a non-trivial element $s_u \in G_u$ for every $u \in V(\Gamma)$. Then, for every word $g$, $\{gs_u \mid u \in V(\Gamma)\}$ is a set of neighbours of $g$ such that every clique containing $g$ contains exactly one of these neighbours. We conclude that $\mathrm{QM}(\Gamma, \mathcal{G})$ is constructible, as desired. 
\end{proof}

\addcontentsline{toc}{section}{References}

\bibliographystyle{alpha}
{\footnotesize\bibliography{TranslationLengths}}

\Address

%

\end{document}